\documentclass[letterpaper]{article}
\usepackage[margin = 1in]{geometry}
\usepackage{graphicx,color}
\usepackage[cp850]{inputenc}
\usepackage[T1]{fontenc}
\usepackage[english]{babel}
\usepackage{cite}
\usepackage{rotating}
\usepackage[f]{esvect}
\usepackage{amsmath, amsthm, amssymb}
\usepackage{rotating}
\usepackage{lscape}
\usepackage{mathrsfs}
\usepackage{url}
\def\tablenotes{\bgroup\parfillskip=0pt plus 1fil
\leftskip=0pt\relax \rightskip=0pt
\vskip2pt\footnotesize}
\def\endtablenotes{\vskip1pt\egroup}

\def\sphline{\noalign{\vskip3pt}\hline\noalign{\vskip3pt}}

\newtheorem{theorem}{Theorem}[section]

\newtheorem{definition}[theorem]{Definition}
\newcommand{\sech}{\mathrm{sech}}

\newcommand{\sinc}{\mathrm{sinc} \,}

\allowdisplaybreaks

\begin{document}
\title{Computing Energy Eigenvalues of Anharmonic Oscillators using the Double Exponential Sinc collocation Method}

\author{Philippe Gaudreau, Richard Slevinsky and Hassan Safouhi\footnote{Corresponding author: hsafouhi@ualberta.ca}\\
Mathematical Section, Campus Saint-Jean\\
University of Alberta\\
8406, 91 Street, Edmonton, Alberta T6C 4G9, Canada}

\date{}

\maketitle

\vspace*{0.75cm}
{\bf \large Abstract}. \hskip 0.10cm

A quantum anharmonic oscillator is defined by the Hamiltonian ${\cal H}= -\frac{ {\rm d^{2}}}{{\rm d}x^{2}} + V(x)$, where the potential is given by $V(x) = \sum_{i=1}^{m} c_{i} x^{2i}$ with $c_{m}>0$. Using the Sinc collocation method combined with the double exponential transformation, we develop a method to efficiently compute highly accurate approximations of energy eigenvalues for anharmonic oscillators. Convergence properties of the proposed method are presented.  Using the principle of minimal sensitivity, we introduce an alternate expression for the mesh size for the Sinc collocation method which improves considerably the accuracy in computing eigenvalues for potentials with multiple wells.

We apply our method to a number of potentials including potentials with multiple wells. The numerical results section clearly illustrates the high efficiency and accuracy of the proposed method. All our codes are written using the programming language Julia and are available upon request.

\vspace*{0.5cm}
 Time independent Schr\"odinger equation. Potentials with multiple wells. Sinc collocation method. Double exponential transformation.  Principle of minimal sensitivity.

\clearpage
\section{Introduction}

The one dimensional anharmonic oscillator is of great interest to field theoreticians because it models complicated fields in one-dimensional space-time~\cite{Bender1969}. A complete overview of quantum anharmonic oscillators would lead to a better understanding of the realistic analytic structure of field theory. Moreover, outside the realm of field theory, the one dimensional anharmonic oscillator also provides an approximation to more complicated quantum potentials near a stable stationary point. The study of quantum anharmonic oscillators as potentials in the Schr\"odinger equation has been on the edge of thrilling and exciting research during the past three decades~\cite{Bender-Orszag-78, Weniger-246-133-96, Weniger-Cizek-Vinette-34-571-93, Zamastil-Cizek-Skala-276-39-99, Patnaik-35-1234-87, Burrows1989,Bender-Wu-7-1620-73, Amore2004, Benassi1979, Adhikari-Dutt-Varshni-131-217-88, Datta-Rampal-23-2875-81, Nanayakkara-289-39-01, Macfarlane1999, Okopinska-36-1273-87, Bozzolo-Plastino-24-3113-81, SouzaDutra-Castro-Boshci-51-3480-95, Flessas1981a, Skala-Cizek-Dvorak-Spirko-53-3241-96, Skala-Dvorak-Kaspa-36-2953-97, Tater-20-2483-87, Tater-Turbiner-26-697-93, Chaudhuri-Mondal-43-3241-91, Agrawal-Varma-49-5089-94, Drozdov-28-445-95, Singh-Singh-Singh-148-389-90, Chhajlany-Letov-Malnev-24-2731-91, Znojil1992, Killingbeck1986b, Fernandez1989a, Bessis-Bessis-38-5483-97, Fernandez-Ma-Tipping-39-1605-89}. With advances in asymptotic analysis and symbolic computing algebra, the interest in developing more efficient methods was renewed recently \cite{Safouhi46, Turbiner-25-647-10, Turbiner2005, Barakat-344-411-05, Bellet2005}. The Hamiltonian in the time-independent Schr\"odinger equation is given by ${\cal H}= -\frac{ {\rm d^{2}}}{{\rm d}x^{2}} + V(x)$ for some potential function $V(x)$. In the case of the quantum anharmonic oscillator, the potential $V(x)$ is an even function of the form $V(x) = \sum_{i=1}^{m} c_{i} x^{2i}$ with $c_{m}>0$. Several approaches have been used for the numerical evaluation of the differential eigenvalue problem ${\cal H}\psi =E\psi$.  However, the existing numerical methods are mostly case specific and lack uniformity when faced with a general problem.

In \cite{Bender-Orszag-78, Weniger-246-133-96, Weniger-Cizek-Vinette-34-571-93}, Rayleigh-Schr\"odinger perturbation series are used to evaluate the ground state energy for potentials of the form $V(x)=x^2 +\beta x^{2m}$ for $\beta\in[0,\infty)$ and $m=2,3,4$.  These summations are strongly divergent for $\beta \neq 0$. To sum them efficiently, Pad\'e approximants combined with nonlinear sequence transformations are used. In \cite{Zamastil-Cizek-Skala-276-39-99}, Rayleigh-Schr\"odinger perturbation series are also used to evaluate energies of the ground state and the first excited state for potentials of the form $V(x)=x^2 +\beta x^{4}$. In \cite{Patnaik-35-1234-87}, Rayleigh-Schr\"odinger perturbation series are used to evaluate energies of the ground state and the first four excited states for the Hamiltonian of the form ${\cal H}= -\frac{1}{2}\frac{ {\rm d^{2}}}{{\rm d}x^{2}} + \frac{1}{2} x^2 +\lambda x^4$ in the limits $\lambda \to 0^{+}$ and $\lambda \to \infty$. In \cite{Burrows1989}, exact soluble models are used to construct Rayleigh-Schr\"odinger perturbation series for the eigenvalues of the anharmonic potentials of the form $V(A, E) = \frac{1}{2}Ax^2 + Ex^4$. In \cite{Bender-Wu-7-1620-73, Amore2004, Benassi1979}, a study of Rayleigh-Schr\"odinger perturbation series is presented using the Wentzel-Kramers-Brillouin~(WKB) method and a difference equation method.

 In \cite{Adhikari-Dutt-Varshni-131-217-88}, an averaging method is proposed to calculate energy eigenvalues for potentials of the form  $V(x)=\lambda x^{2m}$ for $m=2,3,\ldots$ with $\lambda>0$, $V(x)= \mu x^2 + \lambda x^{4} + \eta x^6$ with $\eta>0$ and $V(x)= (ax^3 +bx)^2$ using a supersymmetric WKB approach. In \cite{Datta-Rampal-23-2875-81}, the first few terms of the asymptotic expansion for the energy eigenvalues of the potential $V(x)= ax^2 + bx^{4} + cx^6$ as $n\to\infty$ and in the large coupling limit $c\to\infty$ is found. In \cite{Nanayakkara-289-39-01}, an asymptotic expansion is presented for the energy values of potentials of the form  $V(x)=\sum_{i=}^{N} a_{i}x^{i} + \sum_{j=1}^{M} c_{j}x^{-j}$. This method allows for an easier way to obtain analytically the coefficients for the leading terms in the WKB expansion, which normally would require computation of a considerably large number of complicated contour integrals. As an example of application, the first seven coefficients of the WKB expansion for the energy eigenvalues of the potentials $V(x) = x^4 + bx^2$ and $V(x) = x^6$ are presented. In \cite{Macfarlane1999}, the use of the WKB method and the Lanczos algorithm is used to calculate energy eigenvalues of the potential $V(x) = \frac{1}{2} x^2 + \lambda \,x^{2m}$ to a high accuracy. In \cite{Okopinska-36-1273-87}, the variational principle is used to calculate the first $n$ energy eigenvalues using a Rayleigh-Ritz matrix for the perturbed Hamiltonian ${\cal H}= \frac{1}{2}\frac{ {\rm d^{2}}}{{\rm d}x^{2}} + \frac{1}{2} \Omega^2 x^2 + \epsilon [\lambda x^4 +\frac{1}{2} ( \omega^2 -\Omega^2)x^2]$. Due to optimization of the nonlinear parameter $\Omega$ the method  is  applicable when the parameter $\omega^2<0$ and enables accurate determination of  many lowest eigenvalues even in deep double well potentials. In \cite{Bozzolo-Plastino-24-3113-81}, a variational approach is used to obtain the energies of the ground state and the first excited state of potentials of the form $V(x)=x^2 +\lambda x^{2m}$ for $m=2,\ldots,6$, as well as the energies of the ground state for the potential $V(x)=x^2 +\lambda x^{4} + \mu x^6$. In \cite{SouzaDutra-Castro-Boshci-51-3480-95}, approximate analytic expressions for the energy eigenvalues of the anharmonic oscillator $V(x) =Ax^6 + Bx^2$ is introduced. These approximate solutions were derived from particular analytic solutions which are valid when certain relations between the parameters $A$ and $B$ are satisfied. In \cite{Flessas1981a}, exact solutions in the form of definite integrals are found for the anharmonic oscillator of the form $V(x) = \frac{1}{2} \omega^2 x^2 + \frac{1}{4} \lambda x^4 +\frac{1}{6} \eta x^6$.

The Hill determinant method as well as the Hill determinant method with a variational parameter have also shown great promises \cite{Tater-Turbiner-26-697-93, Chaudhuri-Mondal-43-3241-91, Agrawal-Varma-49-5089-94, Drozdov-28-445-95, Singh-Singh-Singh-148-389-90, Chhajlany-Letov-Malnev-24-2731-91, Znojil1992}. To create a Hill determinant matrix, one starts by substituting a suitable ansatz into the Schr\"{o}dinger equation of the from $\psi(x) = e^{s(x)}\sum_{i=0}^{\infty} \gamma_{i} x^{2i + \delta}$, where $s(x)$ is an even polynomial function that depends on the potential $V(x)$ and $\delta = 0 \, \, {\rm or} \, \, 1$ depending on the parity of the solution. With this substitution, one obtains a recurrence relation for the coefficients $\gamma_{i}$. By rewriting this recurrence relation in a matrix form and setting the determinant of this matrix to zero, one can create a sequence of approximations for the energy eigenvalues of the potential $V(x)$. In \cite{Killingbeck1986b}, the discretization of the Hamiltonian operator using a finite difference technique is discussed to solve this type of eigenvalue problem.

The Riccati equation for the logarithmic derivative of the wavefunction using Pad\'{e} approximants or the Turbiner method has also been used extensively. Further analysis of the Riccati equation solution leads to a better understanding of the overall nature of the wavefunction and thus its energy eigenvalues. In \cite{Fernandez1989a}, a non-perturbative method utilizing the solution to the Riccati equation is proposed for finding energy eigenvalues. The method is applied to the potential $V(x)=x^2 + \lambda x^4$ and yields good estimates for lower energy values. In \cite{Bessis-Bessis-38-5483-97}, a method utilizing the solution to the Riccati equation for finding exact solutions to anharmonic oscillators is discussed. The method is applied to the potential $V(x)=x^2 + \lambda x^4$ and energy eigenvalues are computed for coupling constants ranging from  $\lambda=0.002$ to $\lambda =20000$. In \cite{Fernandez-Ma-Tipping-39-1605-89}, a method is introduced based on rational approximations to the solution of the Riccati equation to obtain tight lower and upper bounds for the energy eigenvalues of anharmonic oscillators.

Recently~\cite{Safouhi46}, an asymptotic expansion for the energy eigenvalues of the potential $V(x) \,=\, \kappa \, x^{2q} + \omega \, x^2,$ where $\kappa \in \mathbb{R}^{+}$, $\omega \in \mathbb{R}$ and $q \in \mathbb{N} \setminus \{1\}$ as the energy level $n$ approaches infinity is derived using the WKB method and series reversion. In \cite{Turbiner2005, Turbiner-25-647-10}, the potential $V(x)=ax^2 + \lambda x^4$ with $a<0$ and $\lambda \geq 0$ is explored. Using asymptotic expansion of the Riccati equation solutions, an approximate solution is found which yields 9-10 significant digits for energy values. In \cite{Barakat-344-411-05}, an asymptotic iteration method is used to calculate the energy eigenvalues of potentials of the form $ V(x) = A x^{2\alpha} + B x^2$.

As can be seen by the numerous approaches which have been made to solve this problem, there is a beautiful diversity yet lack of uniformity in its resolution. While several of these methods yield excellent results for specific cases, it would be favorable to have one general method that could handle any anharmonic potential while being capable of efficiently computing approximations of eigenvalues to a high pre-determined accuracy.

The Sinc collocation method (SCM) has been used extensively during the last 30 years to solve many problem in numerical analysis~\cite{Stenger-33-85-79, Stenger-23-165-81, Stenger-121-379-00, Carlson-Dockery-Lund-66-215-97, Amore-39-L349-06, McArthur-Bowers-Lund-3-169-87, ElGamel-Zayed-48-1285-04, Lund-47-571-86, ElGamel-Cannon-Zayed-73-1325-03, Smith-Bogar-Bowers-Lund-28-760-91}. Their applications include numerical integration, linear and non-linear ordinary differential equations, partial differential equations, interpolation and approximations to function derivatives. Recently, combination of the SCM with the double exponential (DE) transformation has sparked great interest~\cite{Sugihara-Matsuo-164-673-04, Tanaka-Sugihara-Murota-78-1553-09}. The double exponential transformation, introduced in~\cite{Takahasi-Mori-9-721-74} yields optimal accuracy for a given number of function evaluations when using the trapezoidal rule in numerical integration~\cite{Mori-Sugihara-127-287-01}. Since its derivation in 1974, many have studied its effectiveness in computing integrals~\cite{Takahasi-Mori-9-721-74, Sugihara-75-379-97}.

In~\cite{Gaudreau-Slevinsky-Safouhi-14}, we used the SCM with the DE transformation to efficiently solve singular Sturm-Liouville eigenvalue problems. In the following work, we apply this method for computing energy eigenvalues of anharmonic oscillators to unprecedented accuracy. The double exponential Sinc collocation method (DESCM) starts by approximating the wave function as a series of weighted Sinc functions. By substituting this approximation in the Schr\"odinger equation and evaluating this expression at several collocation points spaced equally by a specified mesh size $h$, we obtain a generalized eigensystem which can be transformed into a regular eigenvalue problem. For potentials with multiple wells, the existing expression for the (optimal) mesh size $h$ turns out to be not very effective. In such cases, we introduce an alternate mesh size $\hat{h}$ by minimizing the trace of the resulting matrix. This minimization is justified by the principle of minimal sensitivity.

The proposed method has numerous advantages over the existing alternatives. For example, the methods generality allows for its application to a large set of anharmonic potentials and is insensitive to changes in the potential parameters. The method is now shown to be also efficient and accurate when dealing with potentials with multiple wells. In addition, the DESCM has a near-exponential convergence rate. Moreover, the matrices generated by the DESCM have useful symmetric properties which simplify considerably the computation of their eigenvalues.

\clearpage
\section{General definitions and properties}\label{The double exponential Sinc collocation method}
The sinc function is defined by the following expression:
\begin{equation} \label{formula: sinc functions}
\sinc(z) = \dfrac{\sin(\pi z)}{\pi z}, \qquad z \in \mathbb{C}.
\end{equation}

The Sinc function $S(j,h)(x)$ for $h \in \mathbb{R}^{+}$ and $j \in \mathbb{Z}$ is given by:
\begin{equation}
S(j,h)(x) = \sinc \left( \dfrac{x-jh}{h}\right).
\end{equation}

The discrete orthogonality of Sinc functions is given by:
\begin{equation}
S(j,h)(kh)  =  \delta_{j,k}   \qquad \textrm{for} \qquad j,k \in \mathbb{Z},
\end{equation}
where $\delta_{j,k}$ is the Kronecker's delta function.

Similarly to Fourier series, we can expand well-defined functions as series of Sinc functions. Such expansions are known as Sinc expansions or Whittaker Cardinal expansions.
\begin{definition}\cite{Stenger-23-165-81}
Given any function $v(x)$ defined everywhere on the real line and any $h>0$, the Sinc expansion of $v(x)$ is defined by the following series:
\begin{equation}\label{formula: Sinc expansion}
C(v,h)(x) = \sum_{j=-\infty}^{\infty} v_{j,h} S(j,h)(x),
\end{equation}
where $v_{j,h} = v(jh)$.
\end{definition}

The non-symmetric truncated Sinc expansion of the function $v(x)$ is defined by the following series:
\begin{equation}\label{formula: Truncated Sinc expansion}
C_{N}(v,h)(x) = \sum_{j=-N}^{M} v_{j,h} \, S(j,h)(x) \qquad \textrm{for} \qquad N, M \in \mathbb{N}.
\end{equation}
The symmetric truncated Sinc expansion is obtained by taking $M=N$ in the above equation.

In~\cite{Stenger-23-165-81}, a class of functions which is successfully approximated by a Sinc expansion is introduced. We present the definition for this class of functions bellow.
\begin{definition} \cite{Stenger-23-165-81} \label{defintion: Bd function space}
Let $d>0$ and let $\mathscr{D}_{d}$ denote the strip of width $2d$ about the real axis:
\begin{equation}
\mathscr{D}_{d} = \{ z \in \mathbb{C} : |\,\Im (z)|<d \}.
\end{equation}
In addition, for $\epsilon \in(0,1)$, let $\mathscr{D}_{d}(\epsilon)$ denote the rectangle in the complex plane:
\begin{equation}
\mathscr{D}_{d}(\epsilon) = \{z \in \mathbb{C} : |\,\Re(z)|<1/\epsilon, \, |\,\Im (z)|<d(1-\epsilon) \}.
\end{equation}
Let ${\bf B}_{2}(\mathscr{D}_{d})$ denote the family of all functions $g$ that are analytic in $\mathscr{D}_{d}$, such that:
\begin{equation}\label{formula: integral imaginary}
\displaystyle \int_{-d}^{d} | \,g(x+iy)| \, \textrm{d}y \to 0 \quad \textrm{as} \quad x \to \pm \infty \qquad \textrm{and} \qquad  \mathcal{N}_{2}(g,\mathscr{D}_{d}) = \displaystyle \lim_{\epsilon \to 0} \left(  \int_{\partial \mathscr{D}_{d}(\epsilon)}  |\,g(z)|^{2}\, |\textrm{d}z| \right)^{1/2} <\infty.
\end{equation}
\end{definition}

The time independent Schr\"{o}dinger equation is given by:
\begin{equation}\label{formula:Schrodinger equation}
{\cal H} \, \psi(x) \, \,=\,  E \, \psi(x),
\end{equation}
where the Hamiltonian is given by the following linear operator:
$$
{\cal H} = -\dfrac{{\rm d}^2}{{\rm d} x^2} +V(x),
$$
where $V(x)$ is the potential energy function.

In the case of anharmonic oscillators, the potential $V(x)$ is given by:
\begin{equation}\label{formula: anharmonic oscillator}
V(x)= \displaystyle \sum_{i=1}^{m}c_{i}x^{2i} \qquad  \textrm{with} \qquad  c_{m} > 0 \quad \textrm{and} \quad m \in \mathbb{N}\backslash\{1\}.
\end{equation}

The time independent Schr\"{o}dinger equation~\eqref{formula:Schrodinger equation} can be written as the following boundary value problem:
\begin{equation}\label{formula: Schrodinger equation sturm liouville}
- \psi^{\prime \prime}(x) + V(x)\psi(x) = E \psi(x)  \qquad \textrm{with} \qquad \displaystyle \lim_{|x| \to \infty} \psi(x) = 0.
\end{equation}

Equation~\eqref{formula: Schrodinger equation sturm liouville} is similar to the Sturm-Liouville problem to which we applied successfully the DESCM~\cite{Gaudreau-Slevinsky-Safouhi-14}.

As we stated in~\cite{Gaudreau-Slevinsky-Safouhi-14}, Eggert et al.~\cite{Eggert-Jarratt-Lund-69-209-87} demonstrate that applying an appropriate substitution to the boundary value problem~\eqref{formula: Schrodinger equation sturm liouville}, results in a symmetric discretized system when using Sinc expansion approximations. The change of variable they propose is given by:
\begin{equation}
v(x) = \left(\sqrt{ (\phi^{-1})^{\prime} } \, \psi \right) \circ \phi(x) \qquad \Longrightarrow \qquad  \psi(x)  =  \dfrac{ v \circ \phi^{-1}(x)}{\sqrt{ (\phi^{-1}(x))^{\prime}}},
\label{formula: EggertSub}
\end{equation}
where $\phi^{-1}(x)$ a conformal map of a simply connected domain in the complex plane with boundary points $a\neq b$ such as $\phi^{-1}(a)=-\infty$ and $\phi^{-1}(b)=\infty$.

Applying the substitution~\eqref{formula: EggertSub} to~\eqref{formula: Schrodinger equation sturm liouville}, we obtain:
\begin{align}\label{formula: transformed Schrodinger equation}
\hat{\mathcal{H}} \, v(x) & =  - v^{\prime \prime}(x) + \tilde{V}(x) v(x)
\, = \, E (\phi^{\prime}(x))^{2} v(x),
\end{align}
where:
\begin{equation}
\tilde{V}(x)  =  - \sqrt{\phi^{\prime}(x)} \, \dfrac{{\rm d}}{{\rm d} x} \left( \dfrac{1}{\phi^{\prime}(x)} \dfrac{{\rm d}}{{\rm d} x}( \sqrt{\phi^{\prime}(x)})  \right) + (\phi^{\prime}(x))^{2} V(\phi(x)) \quad \textrm{and} \quad \lim_{ |x| \to \infty} v(x) = 0.
\end{equation}

\section{The double exponential Sinc collocation method (DESCM)}

A function $\omega(x)$ decays double exponentially at infinities if there exist positive constants $A, B,\gamma$ such that:
\begin{equation}
|\, \omega(x)| \leq A \exp( -B \exp(\gamma|\,x|)) \qquad \textrm{for} \qquad x\in \mathbb{R}.
\end{equation}

The double exponential transformation is a conformal mapping $\phi(x)$ which allows for the solution of~\eqref{formula: transformed Schrodinger equation} to have double exponential decay at both infinities.

To implement the DESCM, we begin by approximating the solution of~\eqref{formula: transformed Schrodinger equation} by a truncated Sinc expansion~\eqref{formula: Truncated Sinc expansion}.

Inserting~\eqref{formula: Truncated Sinc expansion} into~\eqref{formula: transformed Schrodinger equation}, we obtain the following system of equations:
\begin{align}\label{formula: index solution}
\hat{\mathcal{H}} \, C_{N}(b,h)(x_{k}) & = \displaystyle \sum_{j=-N}^{N} \left[ - \dfrac{{\rm d}^2}{{\rm d} x_{k}^{2}} S(j,h)(x_{k})  + \tilde{V}(x_{k}) S(j,h)(x_{k}) \right] v_{j,h} \\
& = \mathcal{E} \displaystyle \sum_{j=-N}^{N} S(j,h)(x_{k}) (\phi^{\prime}(x_{k}))^{2}  v_{j,h} \qquad \textrm{for} \qquad k = -N, \ldots, N,
\end{align}
where the collocation points $x_{k}=kh$ and $\mathcal{E}$ is an approximation of the eigenvalue $E$ in~\eqref{formula: transformed Schrodinger equation}.

The above equation can be re-written as follows:
\begin{equation}\label{formula: index solution simplified}
\hat{\mathcal{H}} \, C_{N}(v,h)(x_{k}) =  \sum_{j=-N}^{N} \left[ -\dfrac{1}{h^{2}} \, \delta^{(2)}_{j,k} + \tilde{V}(kh) \, \delta^{(0)}_{j,k}\right] v_{j,h} = \mathcal{E}  \displaystyle \sum_{j=-N}^{N} \delta^{(0)}_{j,k} (\phi^{\prime}(kh))^{2}  v_{j,h} \quad {\rm for} \quad k = -N, \ldots, N,
\end{equation}
where $\delta^{(l)}_{j,k}$ are given by~\cite{Stenger-33-85-79}:
\begin{equation}
\delta^{(l)}_{j,k} =  h^{l} \left. \left( \dfrac{d}{d x} \right)^{l} S(j,h)(x) \right|_{x=kh}.
\end{equation}

Equation~\eqref{formula: index solution simplified} can be represented in matrix form as follows:
\begin{align}\label{formula: matrix solution}
\hat{\mathcal{H}} \, {\bf C}_{N}(v,h) &  = {\bf H}{\bf v} \,=\, \mathcal{E} {\bf D}^{2}{\bf v} \quad \Longrightarrow \quad ({\bf H} - \mathcal{E} {\bf D}^{2}){\bf v} \,=\, 0,
\end{align}
where:
\begin{align*}
{\bf v} & = (v(-Nh),\ldots, v(Nh))^{T} \qquad \textrm{and} \qquad {\bf C}_{N}(v,h)  = (C_{N}(v,h)(-Nh), \ldots, C_{N}(v,h)(Nh) )^{T}.
\end{align*}

${\bf H}$ is a $(2N+1) \times (2N+1)$ matrix with entries $H_{j,k}$ given by:
\begin{equation}\label{formula: H components}
H_{j,k} =   -\dfrac{1}{h^{2}} \, \delta^{(2)}_{j,k} + \tilde{V}(kh) \, \delta^{(0)}_{j,k} \qquad {\rm with} \qquad -N \leq j,k \leq N,
\end{equation}
and ${\bf D}^{2}$ is a $(2N+1) \times (2N+1)$ diagonal matrix with entries $D^{2}_{j,k}$ given by :
\begin{equation}\label{formula: D components}
D^{2}_{j,k} =  (\phi^{\prime}(kh))^{2}  \, \delta^{(0)}_{j,k}  \qquad {\rm with} \qquad -N \leq j,k \leq N.
\end{equation}

To obtain nontrivial solutions for~\eqref{formula: matrix solution}, we have to set:
\begin{equation}
\det({\bf H}-\mathcal{E}{\bf D}^{2}) = 0.
\label{formula: matrix solution-DET}
\end{equation}

To find an approximation of the eigenvalues of equation \eqref{formula: transformed Schrodinger equation}, one simply has to solve this generalized eigenvalue problem. The matrix ${\bf D}^{2}$ is symmetric positive definite and the matrix ${\bf H}$ is symmetric.  If there exits a constant $\delta>0$ such that $\tilde{V}(x)\geq \delta^{-1}$, then the matrix ${\bf H}$ is symmetric positive definite.

In \cite[Theorem 3.2]{Gaudreau-Slevinsky-Safouhi-14}, we present the convergence analysis of DESCM which we state here in the case of the transformed Schr\"odinger equation~\eqref{formula: transformed Schrodinger equation}. The proof of the Theorem is given in~\cite{Gaudreau-Slevinsky-Safouhi-14}.

\begin{theorem}\cite[Theorem 3.2]{Gaudreau-Slevinsky-Safouhi-14} \label{corollary: convergence of eigenvalues}
Let $E$ and $v(x)$ be an eigenpair of the transformed Schr\"odinger equation:
\begin{equation}\label{transformed Sturm Liouville}
 - v^{\prime \prime}(x) + \tilde{V}(x) v(x) = \, E (\phi^{\prime}(x))^{2} v(x),
\end{equation}
where:
\begin{equation}
\tilde{V}(x)  =  - \sqrt{\phi^{\prime}(x)} \, \dfrac{{\rm d}}{{\rm d} x} \left( \dfrac{1}{\phi^{\prime}(x)} \dfrac{{\rm d}}{{\rm d} x}( \sqrt{\phi^{\prime}(x)})  \right) + (\phi^{\prime}(x))^{2} V(\phi(x)) \quad \textrm{and} \quad \lim_{ |x| \to \infty} v(x) = 0.
\end{equation}
Assume there exist positive constants $A,B,\gamma$ such that:
\begin{equation}\label{formula: xi growth condtion}
 |v(x)| \leq A \exp( - B \exp(\gamma |x|)) \qquad \textrm{for all} \qquad x \in \mathbb{R},
\end{equation}
and that $v \in {\bf B}_{2}(\mathscr{D}_{d})$ with $d \leq \dfrac{\pi}{2 \gamma}$.

If there is a constant $\delta>0$ such that $\tilde{V}(x)\geq \delta^{-1}$ and the selection of the optimal mesh size $h$ is such that:
\begin{equation}\label{optimal mesh size}
h = \dfrac{W(\pi d \gamma N / B)}{\gamma N},
\end{equation}
where $W(x)$ is the Lambert W function.

Then, there is an eigenvalue $\mathcal{E}$ of the generalized eigenvalue problem satisfying:
\begin{equation}\label{convergence result}
|\mathcal{E}-E|   \leq \vartheta_{v,d} \sqrt{\delta E} \left(\dfrac{N^{5/2}}{\log(N)^{2}} \right)  \exp \left(-  \dfrac{\pi d \gamma N}{\log(\pi d \gamma N/B)} \right) \quad \textrm{as} \quad N\to \infty,
\end{equation}
where $\vartheta_{v,d}$ is a constant that depends on $v$ and $d$.
\end{theorem}
As we can see from the results obtained in Theorem~\ref{corollary: convergence of eigenvalues}, $|\mathcal{E}-E| \to 0 $ as $N\to \infty$ for all energy eigenvalues~$E$.

\section{Anharmonic oscillators}
To implement the DE transformation, we choose a function $\phi$ which would result in the solution of~\eqref{formula: transformed Schrodinger equation} to decay doubly exponentially.

Since the anharmonic potential is analytic in $\mathbb{C}$ and grows to infinity as $x\to\pm\infty$, the wave function is also analytic in $\mathbb{C}$ and normalizable over~$\mathbb{R}$. More specifically, a simple application of WKB method to equation ~\eqref{formula: Schrodinger equation sturm liouville} with the anharmonic oscillator potential shows that $\psi(x)$ has the following decay rate at both infinities:
\begin{equation}\label{formula: wave function asymptotic}
\psi(x) = {\cal O} \left( |x|^{-m/2} \exp \left( - \dfrac{ \sqrt{c_{m}}\, |x|^{m+1}}{m+1} \right) \right) \qquad \textrm{as} \qquad |x| \to \infty.
\end{equation}
Away from both infinities, the wave function $\psi(x)$ will undergo oscillatory behavior.

As we can see from~\eqref{formula: wave function asymptotic}, the wave function $\psi(x)$ decays only single exponentially at infinities. By taking $\phi(x)=\sinh(x)$, we have:
\begin{eqnarray}\label{formula: wave function asymptotic 2}
|v(x)| & = & \left| \dfrac{\psi \circ \phi(x)}{\sqrt{\phi^{\prime}(x)}} \right| \nonumber\\
& \leq & A |\sinh(x)|^{-m/2}|\cosh(x)|^{-1/2}  \exp \left( - \dfrac{ \sqrt{c_{m}} |\sinh(x)|^{m+1}}{m+1} \right) \nonumber\\
& \leq &  A \exp \left(  - \dfrac{ \sqrt{c_{m}}}{(m+1)2^{m+1}} \exp((m+1)|x|) \right),
\end{eqnarray}
for some positive constant $A$.

From~\eqref{formula: wave function asymptotic 2}, it follow that the optimal mesh size according to Theorem~\ref{corollary: convergence of eigenvalues} is given by:
\begin{equation}\label{formula: Anharmonic optimal mesh size}
h = \dfrac{W \left( \frac{2^{m} \pi^{2}(m+1)N}{\sqrt{c_{m}}} \right)}{(m+1)N}.
\end{equation}

As will be illustrated in our numerical study, the mesh size $h$ given by~\eqref{formula: Anharmonic optimal mesh size} does not prove optimal when working with potentials with multiple wells. In these cases, we can use the principle of minimal sensitivity \cite{Amore-39-L349-06} to obtain an alternate mesh size. First, we will start by simplifying the eigensystem~\eqref{formula: matrix solution} as follows.

Applying a Cholesky factorization to the symmetric positive diagonal matrix ${\bf D}^{2}$, leads to:
\begin{equation}
{\bf D}^{2} = {\bf D}{\bf D}^{T} = {\bf D}{\bf D}.
\label{formula: det simply000}
\end{equation}
Using the above equation, we can re-write the eigensystem~\eqref{formula: matrix solution} as follows:
\begin{eqnarray}\label{formula: generalized eigenvalue problem. schrodinger}
({\bf D}^{-1}{\bf H}{\bf D}^{-1} - \mathcal{E} \, {\bf I} ){\bf z} = 0 \qquad \textrm{and} \qquad  {\bf z} = {\bf Dv}.
\end{eqnarray}
The inverse matrix ${\bf D}^{-1}$ exists since ${\bf D}^{2}$ is a diagonal positive definite matrix.

Let us denote the new matrix in~\eqref{formula: generalized eigenvalue problem. schrodinger} by ${{\bf K}} = {\bf D}^{-1}{\bf H}{\bf D}^{-1}$. Therefore,  ${{\bf K}}$ is a $(2N+1) \times (2N+1)$ matrix with entries ${K}_{j,k}$ given by:
\begin{equation}\label{formula: K components Schrodinger}
K_{j,k} =   -\left(\dfrac{1}{h^{2}\phi^{\prime}(jh) \phi^{\prime}(kh) }\right)\delta^{(2)}_{j,k} \,  + \left(\dfrac{\tilde{V}(kh) \, }{(\phi^{\prime}(kh))^2 } \right) \delta^{(0)}_{j,k} \quad {\rm with} \quad -N \leq j,k \leq N,
\end{equation}
where $\phi(x) =\sinh(x)$ and $\tilde{V}(x)$ is given by:
\begin{equation}
\tilde{V}(x) \,=\, \dfrac{1}{4} - \dfrac{3}{4} \, \sech^{2}(x) + \cosh^{2}(x) \displaystyle \sum_{i=1}^{m}c_{i}\sinh^{2i}(x).
\end{equation}

Denoting the trace of a matrix by ${\rm Tr}(\cdot)$, we have:
\begin{equation}\label{formula: Trace K}
{\rm Tr}({\bf K})(h) = \sum_{i=0}^{2N} \mathcal{E}_{i}(h),
\end{equation}
where $\{\mathcal{E}_{i}(h)\}_{i=0, \ldots, 2N}$ are the $2N+1$ eigenvalues of the matrix ${\bf K}$ or equivalently the generalized eigenvalues of the matrices ${\bf H}$ and ${\bf D}^{2}$. Note that the eigenvalues depend strongly on the mesh size $h$. Since our goal is to obtain the best approximations to these energy eigenvalues, by the principle of minimal sensitivity \cite{Amore-39-L349-06}, it seems logical to minimize their sum with respect to $h$. In other words, this alternate mesh size is given as the solution of the following optimization problem:
\begin{equation}\label{minimization problem}
\hat{h} = \arg \min_{h\in\mathbb{R^{+}}}\{ {\rm Tr}({\bf K})(h) \}.
\end{equation}

As an example, in Figure \ref{figure: traceH}, we plot ${\rm Tr}({\bf K})(h)$ with $N=20$ for the potentials of~\eqref{formula: true value energy} along with the absolute error obtained when approximating energy eigenvalues.
\begin{figure}[!ht]
\begin{center}
\begin{tabular}{cc}
\includegraphics[width=0.4\textwidth]{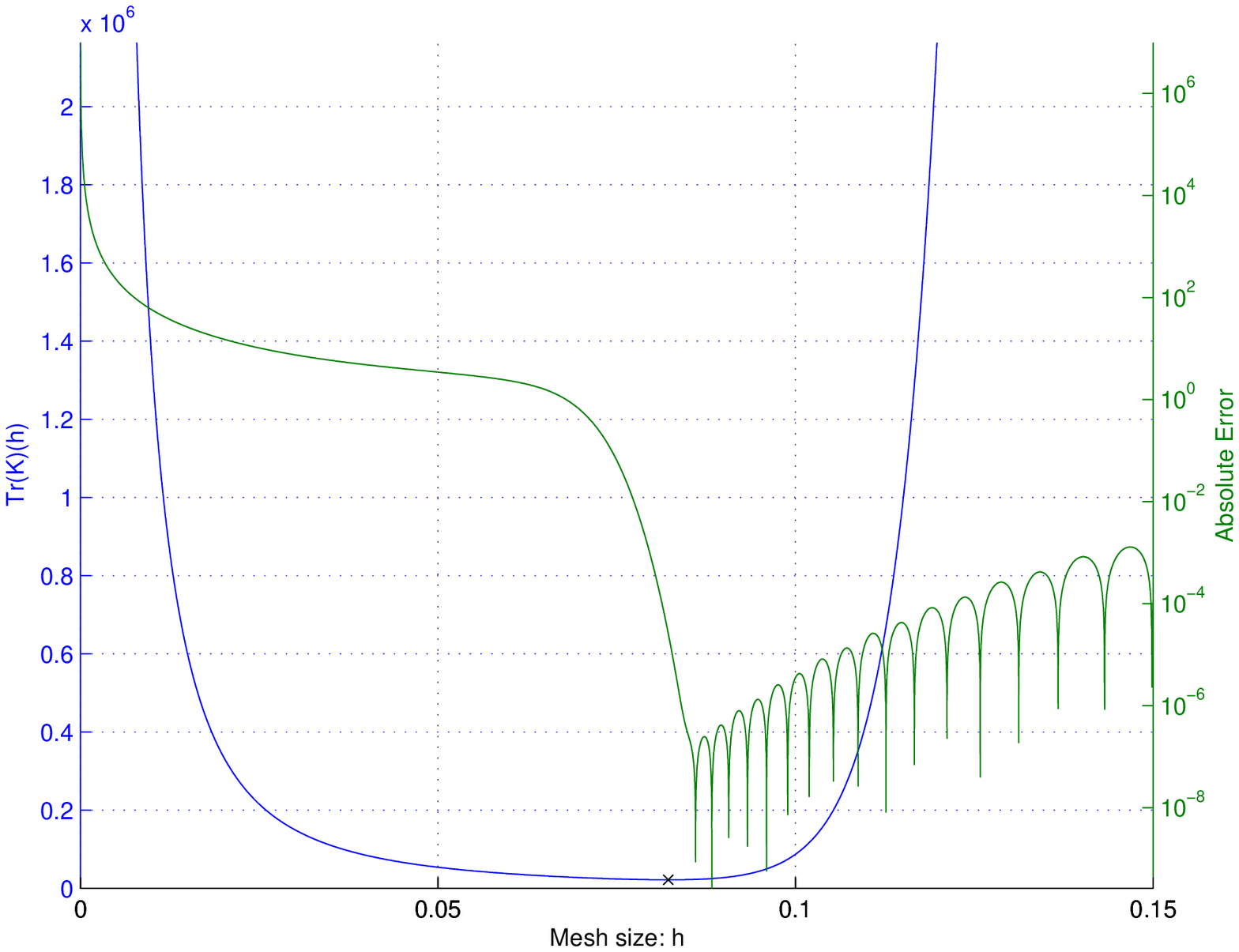}  & \includegraphics[width=0.4\textwidth]{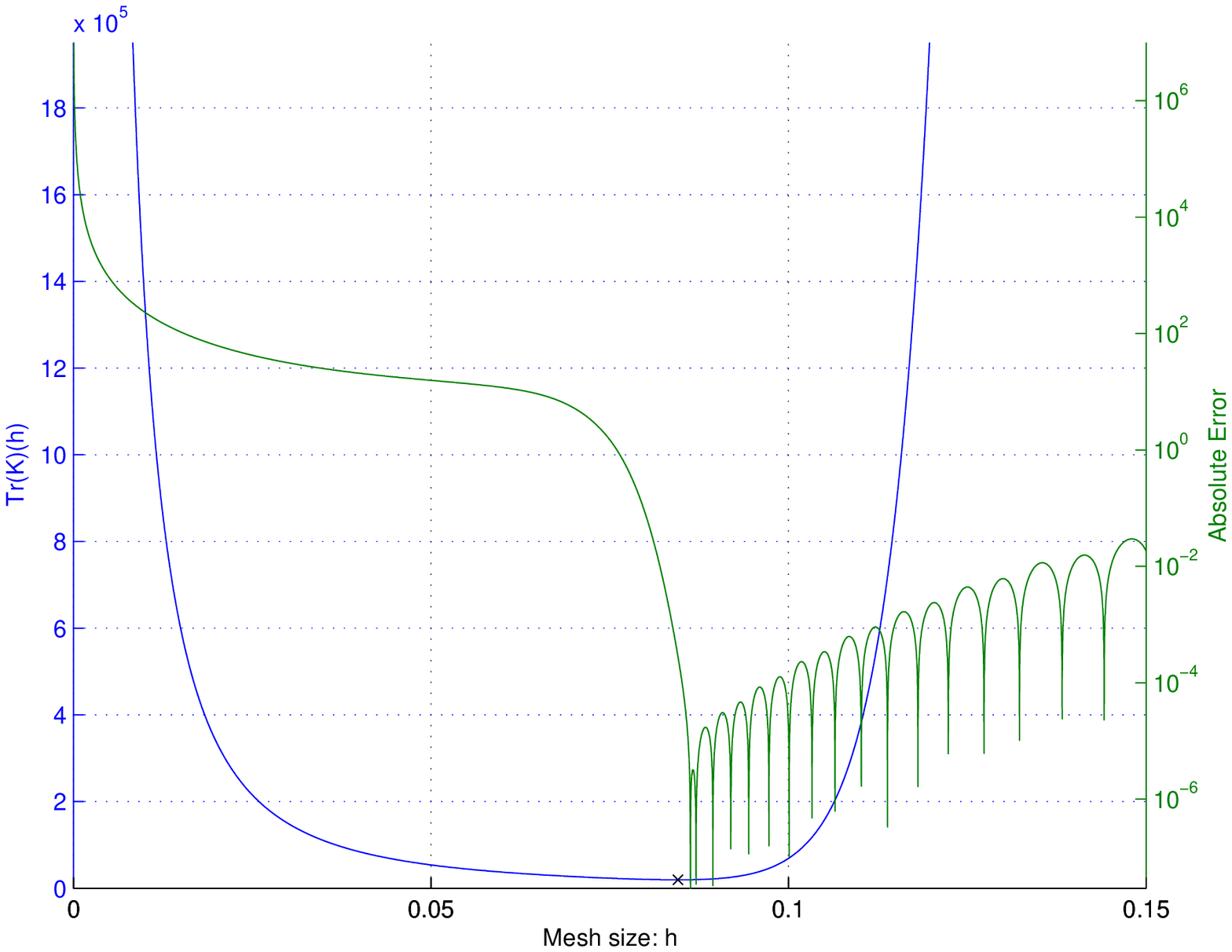} \\
 (a) & (b)\\
 \includegraphics[width=0.4\textwidth]{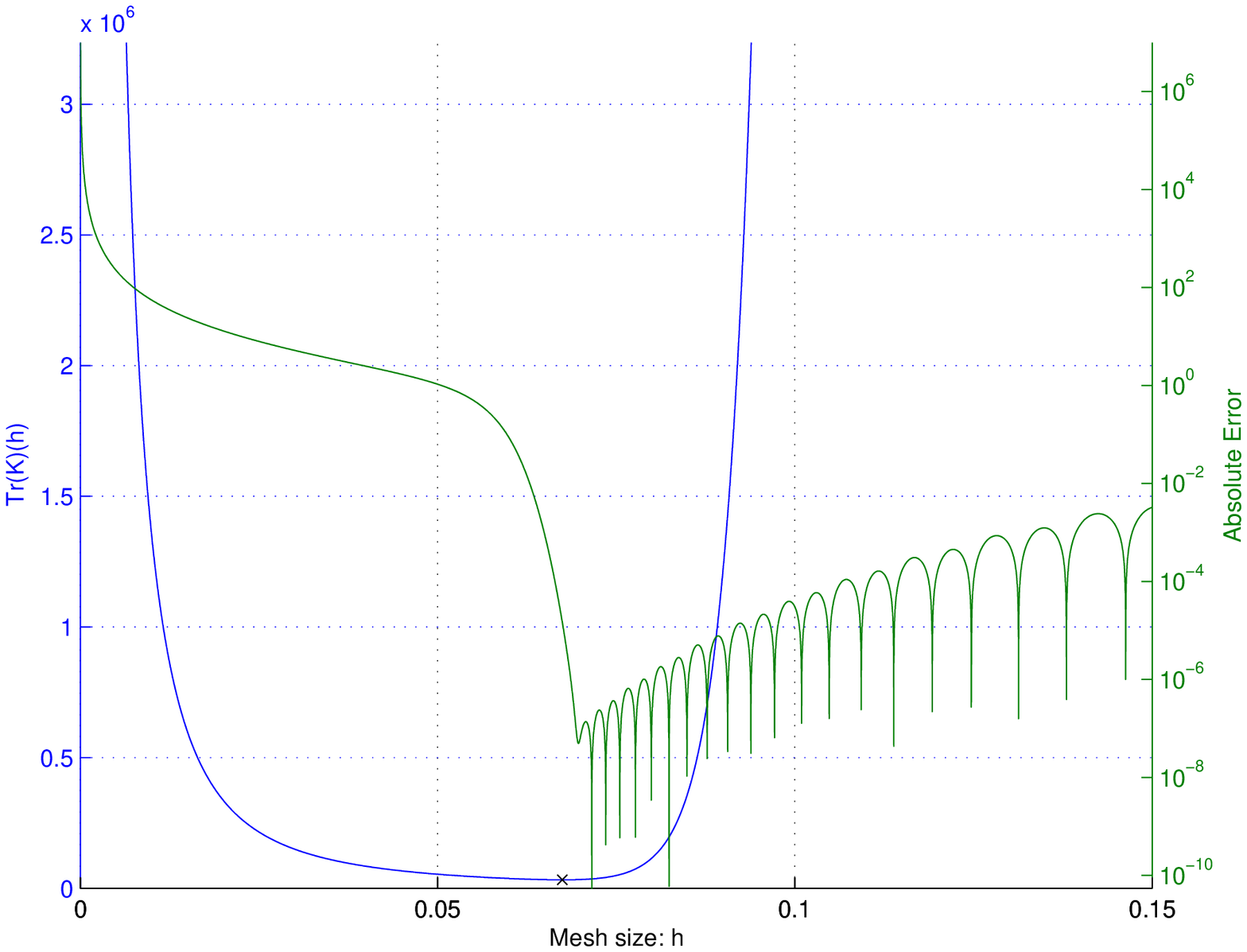}  & \includegraphics[width=0.4\textwidth]{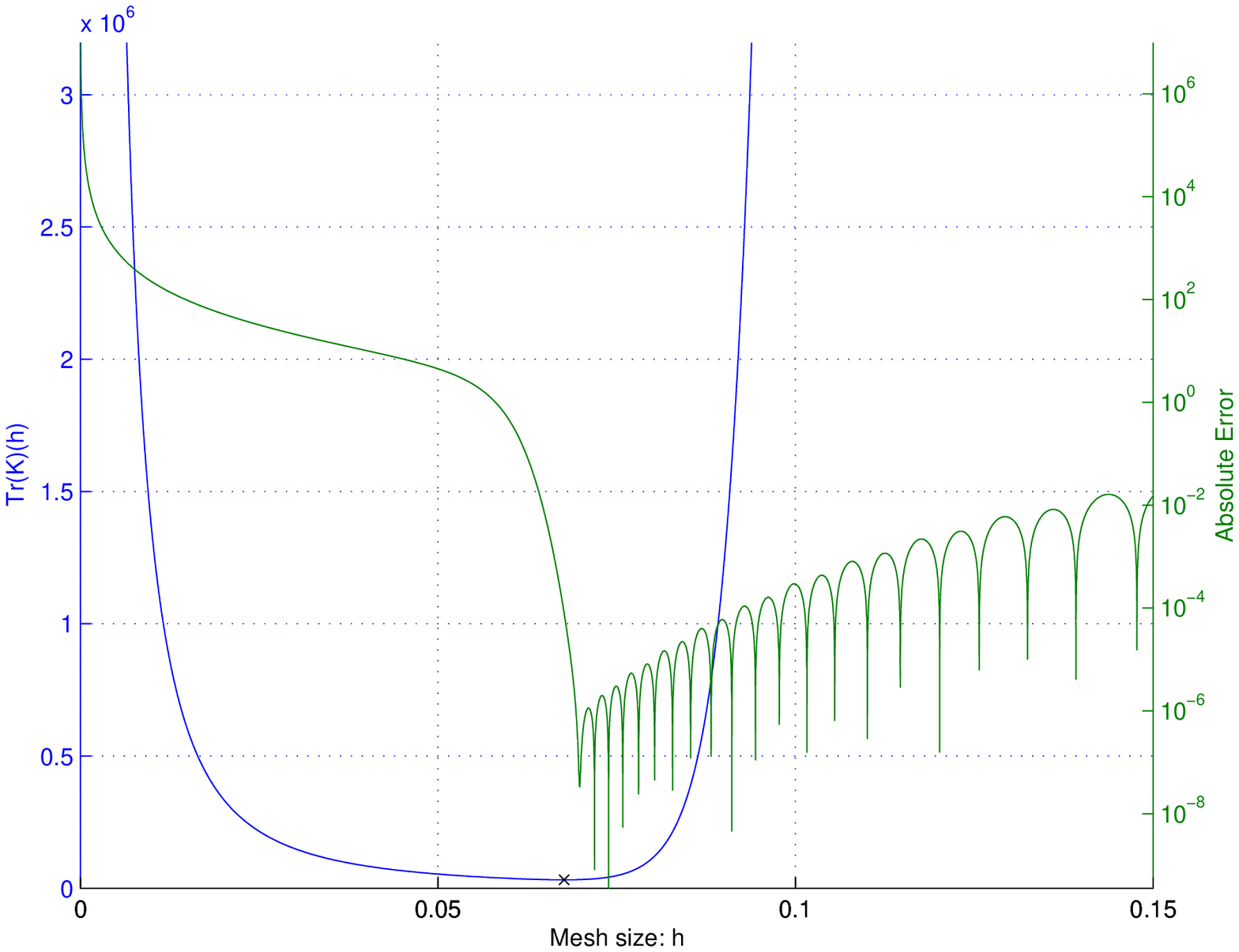} \\
 (c) & (d)
\end{tabular}
\caption{Trace and absolute error vs. $h$ for the potentials $V_{i}(x)$ for $i=1,2,3,4$ as shown in equation \eqref{formula: true value energy} with $\phi(x) = \sinh(x)$. Figure (a) shows the results for the potential $V_{1}(x)  =  x^2 -4x^4+x^6$ with exact eigenvalue $E_{0} = -2$. Figure (b) shows the results for the potential $V_{2}(x)  =  4x^2 -6x^4+x^6$ with exact eigenvalue $E_{1} = -9$. Figure (c) shows the results for the potential $V_{3}(x)  =  (105/64) x^2-(43/8)x^4 + x^6 -x^8 +x^{10}$ with exact eigenvalue $E_{0} = 3/8$. Figure (d) shows the results for the potential $V_{4}(x)  =  (169/64)x^2 -(59/8)x^4 + x^6 -x^8 + x^{10} $ with exact eigenvalue $E_{1} = 9/8$. }
\label{figure: traceH}
\end{center}
\end{figure}

To find this alternate mesh size, one would need to solve the minimization problem in equation \eqref{minimization problem}. To achieve this goal, we require the following theorem establishing the existence of such a minimum.
\begin{theorem}\label{theorem: h min}
If $\,{\bf K}$ is a matrix with components defined by equation \eqref{formula: K components Schrodinger}, $\phi(x)$ is the inverse function of the conformal map $\phi^{-1}(x)$ and $(V(x), \phi(x))\in \mathcal{X}$ where $\mathcal{X}$ is defined as the following function space:
\begin{equation}
\mathcal{X} = \left\{ (V(x) , \phi(x)) \in C(\mathbb{R})\times C^{3}(\mathbb{R}): \lim_{|x|\to \infty} \dfrac{\tilde{V}(x)}{(\phi^{\prime}(x))^{2}}  =  \infty \quad \textrm{and} \quad \phi^{\prime}(x)  >  0, \;\forall x \in  \mathbb{R} \cup \{ \pm \infty\} \right\},
\end{equation}
then for $N\geq1$, $\exists \, \hat{h} \in (0,\infty)$ such that  $ \hat{h} = \displaystyle \arg \min_{h\in\mathbb{R^{+}}}\{ {\rm Tr}({\bf K})(h) \}$.
\end{theorem}

\begin{proof}
The trace of the matrix ${\bf K}$ is given by:
\begin{eqnarray}\label{formula: trace}
\displaystyle {\rm Tr}({\bf K})(h) & = & \dfrac{\pi^2}{3h^{2}} \sum_{k=-N}^{N}\dfrac{1}{(\phi^{\prime}(kh))^{2}}  +   \sum_{k=-N}^{N} \dfrac{\tilde{V}(kh)}{(\phi^{\prime}(kh))^{2}}.
\end{eqnarray}

The function ${\rm Tr}({\bf K})(h) $ is continuous on the interval $(0,\infty)$ because it is composed of continuous functions and $\phi^{\prime}(x) > 0,  \, \, \forall x \in  \mathbb{R}$ by assumption. In addition, the function $\tilde{V}(x)$ is bounded when $x=0$ using this same assumption.

Taking the limit as $h \to 0^{+}$, we obtain:
\begin{eqnarray}
\displaystyle \lim_{h\to 0^{+}} {\rm Tr}({\bf K})(h) & = & \displaystyle \lim_{h\to 0^{+}} \dfrac{\pi^2}{3h^{2}} \displaystyle \sum_{k=-N}^{N} \dfrac{1}{(\phi^{\prime}(kh))^{2}} + \lim_{h\to 0^{+}} \displaystyle \sum_{k=-N}^{N} \dfrac{\tilde{V}(kh)}{(\phi^{\prime}(kh))^{2}}
\nonumber \\  & = & \infty \times \dfrac{(2N+1)}{(\phi^{\prime}(0))^{2}} + \displaystyle  \dfrac{(2N+1)\tilde{V}(0)}{(\phi^{\prime}(0))^{2}}
\nonumber \\ & = & \infty.
\end{eqnarray}

Taking the limit as $h \to \infty$, we obtain:
\begin{eqnarray}
\displaystyle \lim_{h\to \infty } {\rm Tr}({\bf K})(h) & = & \displaystyle \lim_{h\to \infty} \displaystyle  \sum_{k=-N}^{N} \left( \dfrac{\pi^2}{3h^{2}} \dfrac{1}{(\phi^{\prime}(kh))^{2}} \right) + \lim_{h\to \infty} \displaystyle  \sum_{k=-N}^{N} \dfrac{\tilde{V}(kh)}{(\phi^{\prime}(kh))^{2}}
\nonumber \\ & & \hskip -2.5cm = \, 0+  \dfrac{\tilde{V}(0)}{\rho(\phi(0))(\phi^{\prime}(0))^{2}}
+ \, N \times \left( \displaystyle \lim_{x\to \infty }  \dfrac{\tilde{V}(x)}{(\phi^{\prime}(x))^{2}} + \displaystyle \lim_{x\to - \infty }  \dfrac{\tilde{V}(x)}{(\phi^{\prime}(x))^{2}} \right)
\nonumber \\ & & \hskip -2.5cm = \, \infty.
\end{eqnarray}

Since:
$$
\lim_{h\to 0^{+}} {\rm Tr}({\bf K})(h) =  \lim_{h\to \infty } {\rm Tr}({\bf K})(h) = \infty,
$$
and the function ${\rm Tr}({\bf K})(h)$ is continuous on the interval $(0,\infty)$, by the Weierstrass extreme value theorem, $\exists \, \hat{h} \in (0,\infty)$ such that $ \hat{h} = \displaystyle \arg \min_{h\in\mathbb{R^{+}}}\{ {\rm Tr}({\bf K})(h) \}$.
\end{proof}

By construction, we already know that ${\bf K}$ is a symmetric matrix. However, without loss of generality, we can also assume the matrix ${\bf K}$ to be positive definite by the following demonstration.

Given a potential of the form in~\eqref{formula: anharmonic oscillator}, it is possible to find a constant $\Omega > 0$ such that $\tilde{V}(x) > -\Omega\cosh^{2}(x)$ for all $x \in \mathbb{R}$. Consequently, we can rewrite~\eqref{formula: transformed Schrodinger equation} as follows:
\begin{align}
- v^{\prime \prime}(x) + \tilde{V}(x) v(x)  & =   E \cosh^{2}(x) v(x) \nonumber\\
&   \Longrightarrow \; - v^{\prime \prime}(x) + \tilde{V}(x)v(x) + \Omega \cosh^{2}(x)v(x)  \,=\, E \cosh^{2}(x) v(x) + \Omega \cosh^{2}(x)v(x) \nonumber\\
&   \Longrightarrow \; - v^{\prime \prime}(x) + (\tilde{V}(x) + \Omega \cosh^{2}(x))v(x)  \,=\, (\Omega + E) \cosh^{2}(x) v(x) \nonumber\\
&   \Longrightarrow \; - v^{\prime \prime}(x) + \hat{V}(x)v(x)  \,=\,  \hat{E} \cosh^{2}(x) v(x),
\end{align}
where $\hat{V}(x) = \tilde{V}(x) + \Omega\cosh^{2}(x)>0$ and $\hat{E}= \Omega + E$.

Since $\hat{V}(x)>0$ for all $x \in \mathbb{R}$, the matrix ${\bf H}$ resulting from the DESCM will be positive definite. Consequently, the matrix ${\bf K}$ will also be positive definite. With this in mind, all the assumptions of Theorem~\ref{corollary: convergence of eigenvalues} are satisfied. Hence, the eigenvalues of the problem~\eqref{formula: generalized eigenvalue problem. schrodinger} converge to the eigenvalues of~\eqref{formula: Schrodinger equation sturm liouville}.

\section{Numerical discussion}
In this section, we present numerical results for the energy values of anharmonic oscillator potentials.

All calculations are performed using the programming language Julia~\cite{Bezanson2012} in double precision. The eigenvalue solvers in Julia utilize the linear algebra package {\it LAPACK}~\cite{laug}. Unless otherwise stated, the mesh size $h$~\eqref{formula: Anharmonic optimal mesh size} is used in all calculations.

In certain cases, we use the mesh size $\hat{h}$ obtained by solving the optimization problem in~\eqref{minimization problem}. In these cases, we use the optimization Julia package {\it Optim} \cite{White2012}. The matrix ${\bf K}$ is constructed using~\eqref{formula: K components Schrodinger}.

In \cite{Chaudhuri1991a}, Chaudhuri et al. presented several potentials which had known analytic solutions for energy levels calculated using supersymmetric quantum mechanics, namely:
\begin{equation}\label{formula: true value energy}
\begin{array}{lllll}
V_{1}(x) & = & x^2 -4x^4+x^6  & \Rightarrow &  E_{0} = -2 \\
V_{2}(x) & = & 4x^2 -6x^4+x^6 & \Rightarrow &  E_{1} = -9 \\
V_{3}(x) & = & (105/64) x^2-(43/8)x^4 + x^6 -x^8 +x^{10} &\Rightarrow &  E_{0} = 3/8 \\
V_{4}(x) & = & (169/64)x^2 -(59/8)x^4 + x^6 -x^8 + x^{10}   &\Rightarrow &  E_{1} = 9/8.
\end{array}
\end{equation}

Using these exact values, we present Figure~\ref{figure: potentials} to illustrate the convergence of the DESCM. Figure~\ref{figure: potentials} shows the absolute error between our approximation and the exact values given in~\eqref{formula: true value energy}. The absolute error is defined by:
\begin{equation}
{\rm Absolute \, \, error} = \left| \mathcal{E}_{l}(N) - {\rm Exact \, \, value} \right| \qquad \textrm{for} \qquad l =0,1.
\end{equation}

As can be seen from Figure~\ref{figure: potentials}, the approximations obtained using DESCM converge quite well.

In Tables~\ref{TABLE: 1} and \ref{TABLE: 2}, we present approximations of energies for the ground state and first two excited states for two different potentials with unknown energy eigenvalues. There appears to be convergence in all cases.

Tables~\ref{TABLE: 3},~\ref{TABLE: 4},~\ref{TABLE: 5}, and~\ref{TABLE: 6} display the ground state energy for various potentials as well as an approximation to the absolute error. In these tables, the approximation to the absolute error is given by:
\begin{equation}\label{formula: approximate relative error}
\epsilon_{n}(N) = \left | \mathcal{E}_{n}(N-1)-\mathcal{E}_{n}(N) \right | \qquad \textrm{for} \qquad N = 2,3,4,\ldots \quad \textrm{and} \quad n = 0,1,2,\ldots.
\end{equation}
In all four of these tables, we use $\epsilon_{0}(N)<5 \times 10^{-12}$ as a stopping criterion.

Table \ref{TABLE: 3} displays values obtained for the potential $V(x)=c_1 x^2 + c_2 x^{4}$ for different values of $c_1$ and $c_2$.
Table \ref{TABLE: 4} displays values obtained for the potential $V(x)=c_1 x^2 + c_2 x^{4}+c_3 x^{6}$ for different values of $c_1$, $c_2$ and $c_3$.
Table \ref{TABLE: 5} displays values obtained for the potential $V(x)=c_1 x^2 + c_2 x^{4}+c_3 x^{6} + c_4x^{8} $ for different values of $c_1$, $c_2$, $c_3$ and $c_4$. Table \ref{TABLE: 6} displays values obtained for the potential $V(x)=c_1 x^2 + c_2 x^{4}+c_3 x^{6} + c_4x^{8} +c_{5} x^{10}$ for different values of $c_1$, $c_2$, $c_3$, $c_4$ and $c_5$. In all these tables, the numbers between parentheses represent powers of ten.

In general, the DESINC method performs well when using the optimal mesh size $h$ for low oscillatory potentials. However, as the number of oscillations increase in the potential, the mesh size $h$ performs significantly less than the step size $\hat{h}$. To illustrate this claim, we present in Figure~\ref{figure: Multiple wells step size} three potentials with three, five and ten wells respectively. The first row in Figure~\ref{figure: Multiple wells step size} illustrates the step sizes used and the convergence of the DESINC method for the three well potential $V_{2}(x)  =  4x^2 -6x^4+x^6$ in \eqref{formula: true value energy} with exact eigenvalue $E_{1} = -9$.  The second row in Figure \ref{figure: Multiple wells step size} displays the step sizes used and the convergence of the DESINC method for the five well potential $V(x)  =  T_{10}(x)-1$, where $T_{10}(x)$ is the $10^{\textrm{th}}$ Chebyshev polynomial. Finally the third row in Figure~\ref{figure: Multiple wells step size} displays the step sizes used and the convergence of the DESINC method for the ten well potential $V(x)  =  T_{20}(x)-1$, where $T_{20}(x)$ is the $20^{\textrm{th}}$ Chebyshev polynomial.

In Figure~\ref{figure: Multiple wells}, we implemented our algorithm with the step size $\hat{h}$ for the ten well potential:
\begin{eqnarray}\label{formula: ten well potential}
V(x) & = &  T_{20}(x)-1,
\end{eqnarray}
for $N=1,2,\ldots,1000$.

Using the stopping criterion $\epsilon_{n}(N)<5 \times 10^{-12}$, we were able to find an approximation to 1353 eigenvalues of this ten well potential.

In~\cite{Weniger-246-133-96}, Weniger uses a Rayleigh-Schr\"odinger perturbation series and sequence transformations to evaluate the ground state of the potential $V(x) = x^2 + x^{4}$ to high accuracy. More specifically, Weniger uses the exact rational arithmetics of Maple with an accuracy of 300 decimal digits to obtain the following value:
\begin{equation}\label{formula: ground Weniger}
E_{0} \approx 1.392~351~641~ 530~ 291~ 855~ 657~ 507~ 876~ 609~ 934~ 184~ 600~ 066~ 711~ 9.
\end{equation}

We used Maple16\textsuperscript{\tiny{TM}} to implement our algorithm for the same potential with an accuracy of 100 correct digits and we obtain:
\begin{eqnarray}\label{formula: ground Us}
E_{0} & \approx & 1.392~ 351~ 641~ 530~ 291~ 855~ 657~ 507~ 876~ 609~ 934~ 184~ 600~ 066~ 711~ 220~ 834~
\nonumber\\ &  & 088~ 906~ 349~ 323~ 877~ 567~ 431~ 875~ 646~ 528~ 590~ 973~ 563~ 467~ 791~ 759~ 121,
\end{eqnarray}
which is in excellent agreement with Weniger's value.

\section{Conclusion}
Various methods have been used to calculate the energy eigenvalues of quantum anharmonic oscillators given a specific set of parameters. While several of these methods yield excellent results for specific cases, there is a beautiful diversity yet lack of uniformity in the resolution of this problem. In this work, we present a method based on the DESCM where the wave function of a transformed Schr\"odinger equation \eqref{formula: transformed Schrodinger equation} is approximated by as a Sinc expansion. By summing over $2N+1$ collocation points, we construct a symmetric positive definite matrices ${\bf K}$ whose eigenvalues are approximations to the energy eigenvalues of~\eqref{formula:Schrodinger equation}. The DESCM method has a convergence rate of ${\cal O}\left(\left(\frac{N^{5/2}}{\log(N)^{2}} \right)  \exp \left(- \kappa \frac{N}{\log(N)} \right) \right)$.  The convergence is improved for potential with multiple wells by using the alternate mesh size obtained by minimizing the trace of the discretized Hamiltonian.

The numerical results obtained for a number of different potentials including potentials with multiple wells, show clearly the efficiency and accuracy of the proposed method.

\section*{Acknowledgment}

The corresponding author (HS) acknowledges the financial support for this research by the Natural Sciences and Engineering Research Council of Canada~(NSERC).

\section{Tables and Figures}

\clearpage

\begin{figure}[!ht]
\begin{center}
\begin{tabular}{cc} \includegraphics[width=0.35\textwidth]{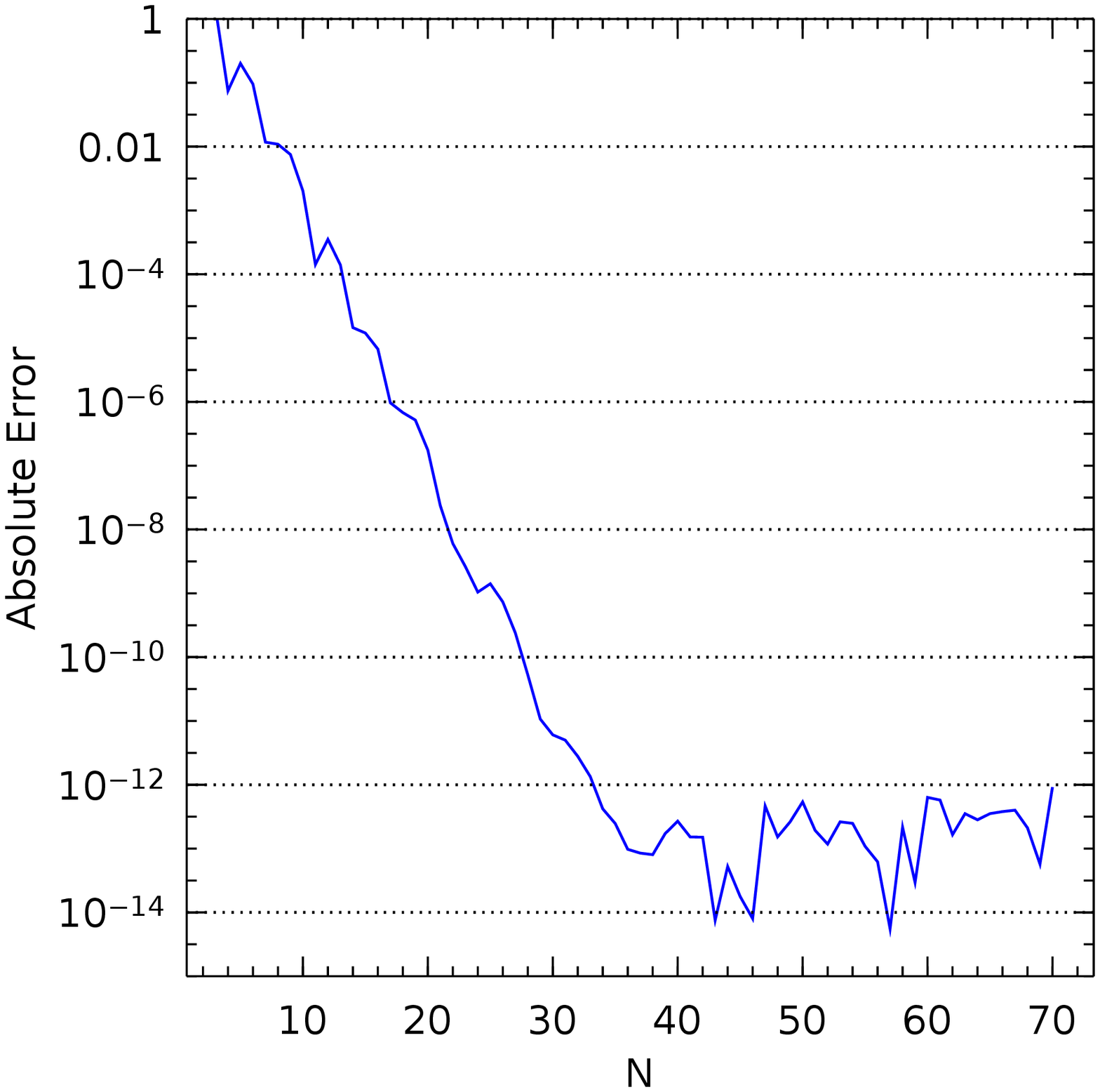} &  \includegraphics[width=0.35\textwidth]{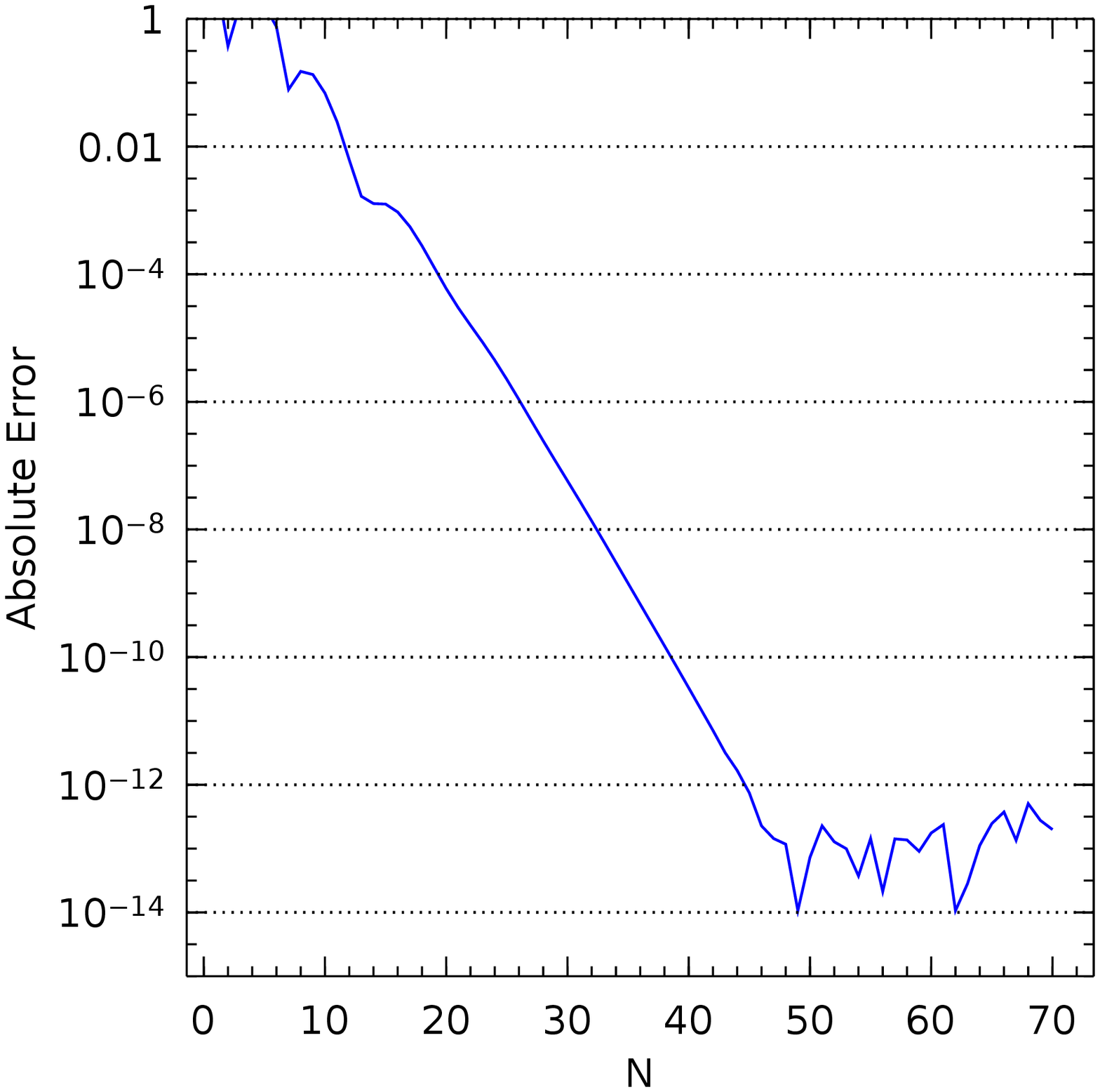} \\
(a) & (b) \\
\includegraphics[width=0.35\textwidth]{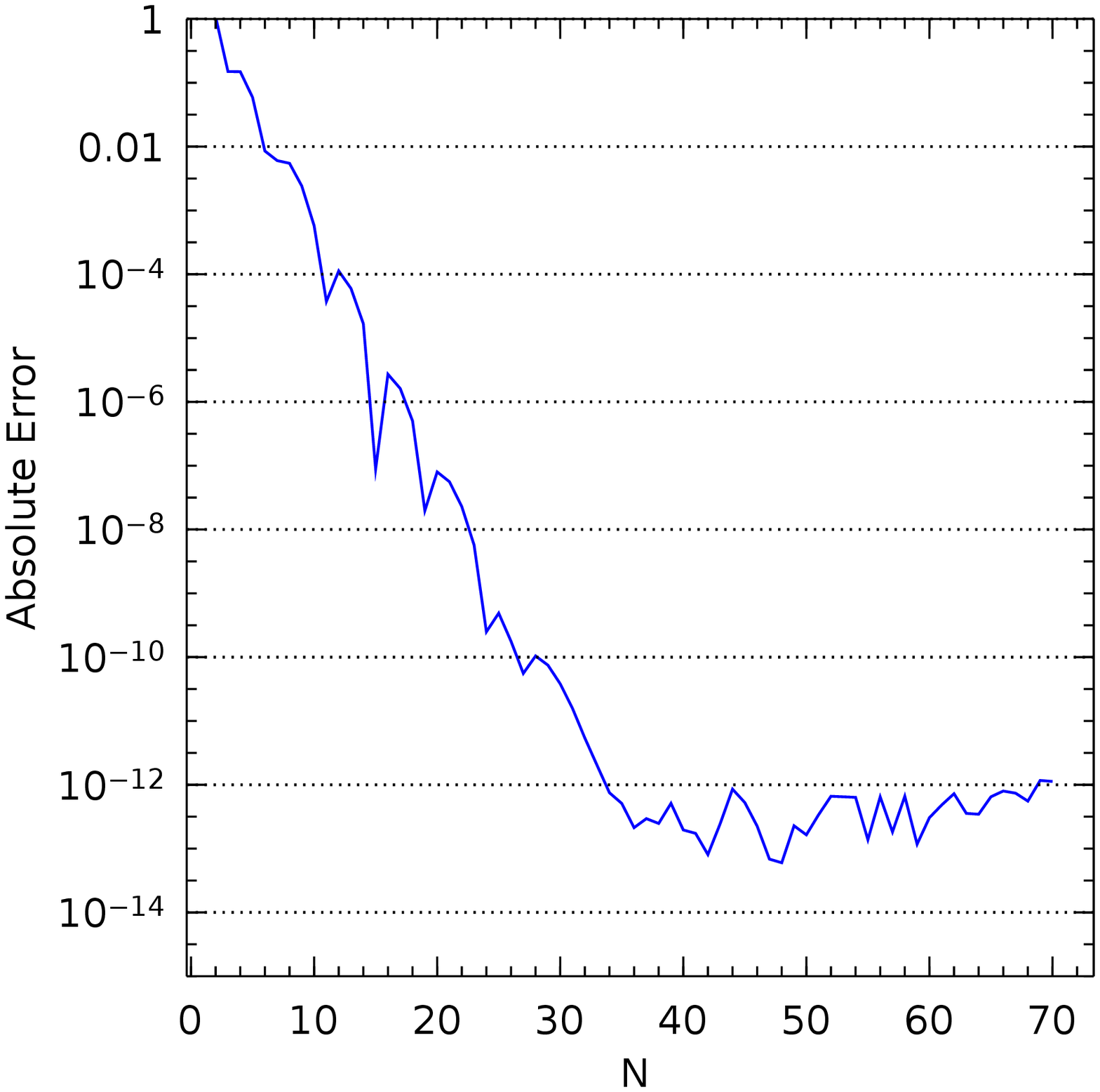} &  \includegraphics[width=0.35\textwidth]{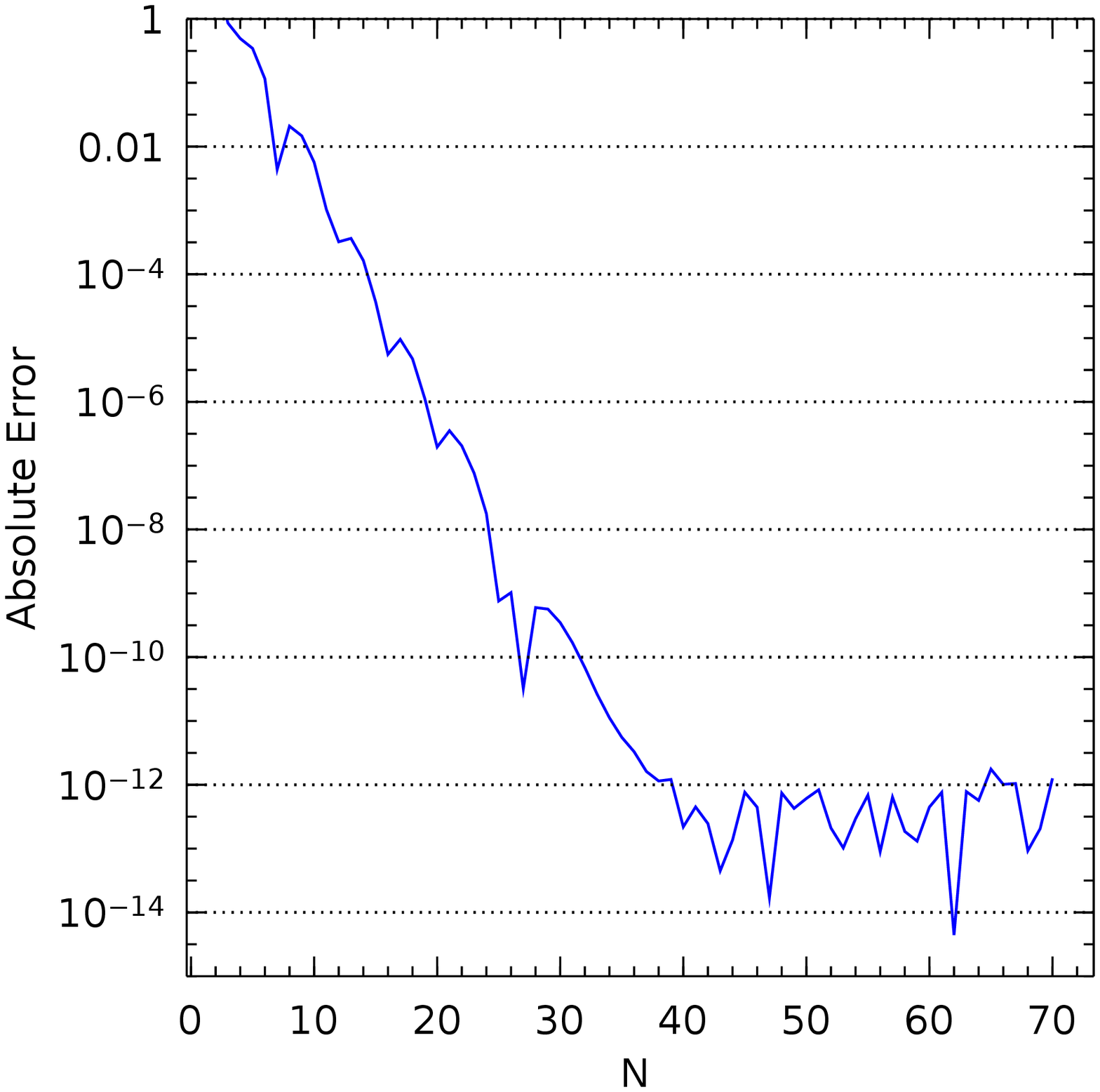} \\
(c) & (d)
\end{tabular}
\caption{Absolute error for the potentials $V_{i}(x)$ for $i=1,2,3,4$ given by~\eqref{formula: true value energy} with $\phi(x) = \sinh(x)$. \newline
(a)~$V_{1}(x) =  x^2 -4x^4+x^6$ with exact eigenvalue $E_{0} = -2$. (b)~$V_{2}(x)  =  4x^2 -6x^4+x^6$ with exact eigenvalue $E_{1} = -9$. (c)~$V_{3}(x)  =  (105/64) x^2-(43/8)x^4 + x^6 -x^8 +x^{10}$ with exact eigenvalue $E_{0} = 3/8$. (d)~$V_{4}(x)  =  (169/64)x^2 -(59/8)x^4 + x^6 -x^8 + x^{10} $ with exact eigenvalue $E_{1} = 9/8$.}
\label{figure: potentials}
\end{center}
\end{figure}

\begin{table}[!ht]
\caption{Energies for the ground state and first two excited states for $V(x) = -x^2 + 3 x^4 -2 x^6 +0.1 x^{10}.$ }
\begin{tabular*}{\hsize}{@{\extracolsep{\fill}}cccc} \sphline
  $N$ &   $\mathcal{E}_{0}(N)$        &  $\mathcal{E}_{1}(N)$        &      $\mathcal{E}_{2}(N)$   \\\hline
   ~5 & -0.183054938746611~ & 0.441479870018253  &  2.620487757023682\\
   10 & -0.0976947154532108 & 0.670920848438211  &  3.112803149372351\\
   15 & -0.0962838618463357 & 0.672983395806946  &  3.110900059783247\\
   20 & -0.0962939179110841 & 0.672989564944146  &  3.111020042497232\\
   25 & -0.0962917320927764 & 0.672993682058299  &  3.111022843861247\\
   30 & -0.0962919468261398 & 0.672993241672601  &  3.111022328272051\\
   35 & -0.0962919458832259 & 0.672993243476173  &  3.111022329656410\\
   40 & -0.0962919462260392 & 0.672993242754209  &  3.111022328736961\\
   45 & -0.0962919462302011 & 0.672993242746560  &  3.111022328725989\\
   50 & -0.0962919462309655 & 0.672993242745170  &  3.111022328724715\\ \hline
\end{tabular*}
\label{TABLE: 1}
\end{table}

\begin{table}[!ht]
\caption{Energies for the ground state and first two excited states for $V(x) = x^2 + 100x^{8}$. }
\begin{tabular*}{\hsize}{@{\extracolsep{\fill}}cccc} \sphline
  $N$ &      $\mathcal{E}_{0}(N)$    &  $\mathcal{E}_{1}(N)$        &      $\mathcal{E}_{2}(N)$   \\\hline
   ~3 & 3.18583889990311 & 12.1774056576440  &  25.9667305118017\\
   ~6 & 3.18865215097014 & 12.1950090976147  &  26.0334131709351\\
   ~9 & 3.18865434610824 & 12.1950219328947  &  26.0334583310462\\
   12 & 3.18865434649856 & 12.1950219336715  &  26.0334583214430\\
   15 & 3.18865434649231 & 12.1950219336306  &  26.0334583212540\\
   18 & 3.18865434649241 & 12.1950219336298  &  26.0334583212524\\
   21 & 3.18865434649213 & 12.1950219336305  &  26.0334583212523\\
   24 & 3.18865434649426 & 12.1950219336305  &  26.0334583212539\\
   27 & 3.18865434649200 & 12.1950219336299  &  26.0334583212526\\
   30 & 3.18865434649236 & 12.1950219336314  &  26.0334583212516\\ \hline
\end{tabular*}
\label{TABLE: 2}
\end{table}

\begin{table}[!ht]
\caption{The ground state energy for $V(x) = c_1 x^2 + c_2 x^{4}$. }
\begin{tabular*}{\hsize}{@{\extracolsep{\fill}}ccccc} \sphline
   $c_{1}$ & $c_{2}$ & $N$&  $\mathcal{E}_{0}(N)$   &  $\epsilon_{0}(N)$   \\\hline
   ~0.1     & 0.1 &  20 & 5.6694532770815997(-1) &	 1.6(-12) \\
   ~0.1     & ~~1 &  18 & 1.0962243662319233(~0) &	 2.3(-12) \\
   ~~~1     & ~~1 &  17 & 1.3923516415352821(~0) &	 2.5(-12) \\
   ~~~1     & ~10 &  17 & 2.4491740721179220(~0) &	 8.8(-14) \\
   ~~10     & ~10 &  15 & 3.7029004216662731(~0) &	 4.0(-13) \\ \hline
   -0.1     & 0.1 &  21 & 4.1046961591503783(-1) &	 2.6(-12) \\
   -0.1     & ~~1 &  18 & 1.0238094432848113(~0) &	 4.7(-13) \\
   ~~-1     & ~~1 &  19 & 6.5765300518294945(-1) &	 5.4(-14) \\
   ~~-1     & ~10 &  17 & 2.1128778980507850(~0) &	 7.1(-13) \\
   ~-10     & ~10 &  19 & 9.0479065692642441(-2) &	 1.7(-12) \\ \hline
\end{tabular*}
\label{TABLE: 3}
\end{table}

\begin{table}[!ht]
\caption{The ground state energy for $V(x) = c_1 x^2 + c_2 x^{4}+c_3 x^{6}$. }
\begin{tabular*}{\hsize}{@{\extracolsep{\fill}}cccccc} \sphline
   $c_{1}$ & $c_{2}$ & $c_{3}$ & $N$&  $\mathcal{E}_{0}(N)$          &  $\epsilon_{0}(N)$   \\\hline
   ~0.1     & ~0.1     &   0.1   & 23 &	 7.6469531499643029(-1) &	 4.2(-13) \\
   ~~~1     & ~~~1     &   ~~1   & 20 &	 1.6148940820343036(~0) &	 1.6(-12) \\
   ~0.1     & ~~~1     &   ~10   & 19 &	 2.1277742176946535(~0) &	 3.7(-12) \\
   ~~~1     & ~~10     &   ~10   & 17 &	 2.7940871778594101(~0) &	 3.3(-12) \\
   ~~10     & ~~10     &   ~10   & 16 &	 3.8948206179865981(~0) &	 2.5(-12) \\ \hline
   -0.1     & ~0.1     &   0.1   & 23 &	 6.6383017274207901(-1) &	 2.0(-12) \\
   ~~~1     & ~~-1     &   ~~1   & 23 &	 1.2022669303165900(~0) &	 8.0(-13) \\
   -0.1     & ~~-1     &   ~10   & 20 &	 1.9385567907196897(~0) &	 2.7(-13) \\
   ~~-1     & ~~10     &   ~10   & 17 &	 2.5157308558338656(~0) &	 2.3(-12) \\
   ~~10     & ~-10     &   ~10   & 20 &	 2.9588710692969618(~0) &	 1.9(-12) \\\hline
\end{tabular*}
\label{TABLE: 4}
\end{table}

\begin{table}[!ht]
\caption{The ground state energy for $V(x) = c_1 x^2 + c_2 x^{4}+c_3 x^{6} + c_{4} x^{8}$. }
\begin{tabular*}{\hsize}{@{\extracolsep{\fill}}ccccccc} \sphline
   $c_{1}$ & $c_{2}$ & $c_{3}$ & $c_{4}$ & $N$&  $\mathcal{E}_{0}(N)$          &  $\epsilon_{0}(N)$   \\\hline
   ~0.1     & ~0.1     &   ~0.1   &   0.1  & 23 &	~9.2287072386834434(-1) &	 3.0(-13) \\
   ~0.1     & ~~~1     &   ~~10   &   ~10  & 21 &	~2.3988345516957166(~0) &	 2.2(-12) \\
   ~~~1     & ~~~1     &   ~~10   &   ~10  & 21 &	~2.5285749972092857(~0) &	 2.2(-12) \\
   ~~~1     & ~~10     &   ~~10   &   ~10  & 20 &	~2.9458972541841404(~0) &	 9.8(-13) \\
   ~~10     & ~~10     &   ~~10   &   ~10  & 19 &	~3.9840271957255702(~0) &	 3.1(-12) \\\hline
   -0.1     & ~0.1     &   -0.1   &   0.1  & 27 &	~6.9423980434904176(-1) &	 1.6(-12) \\
   ~0.1     & ~~-1     &   ~~10   &   ~10  & 22 &	~2.2867765902246440(~0) &	 1.0(-12) \\
   ~~-1     & ~~-1     &   ~~10   &   ~10  & 22 &	~2.1181378732419969(~0) &	 1.4(-12) \\
   ~~~1     & ~~10     &   ~-10   &   ~10  & 23 &	~2.3756889547019138(~0) &	 3.9(-12) \\
   ~-10     & ~-10     &   ~-10   &   ~10  & 35 &	-9.7139097706403668(~0) &	 4.8(-12) \\ \hline
\end{tabular*}
\label{TABLE: 5}
\end{table}

\begin{table}[!ht]
\caption{The ground state energy for $V(x) = c_1 x^2 + c_2 x^{4}+c_3 x^{6} + c_{4} x^{8}+ c_{5} x^{10}$. }
\begin{tabular*}{\hsize}{@{\extracolsep{\fill}}cccccccc} \sphline
   $c_{1}$ & $c_{2}$ & $c_{3}$ & $c_{4}$&$c_{5}$& $N$&  $\mathcal{E}_{0}(N)$          &  $\epsilon_{0}(N)$   \\\hline
   ~0.1     & ~0.1     &   ~0.1   &   ~0.1  &  0.1  & 27 &	~1.0520482472987258(~0) &	 4.9(-12) \\
   ~0.1     & ~0.1     &   ~~~1   &   ~~~1  &  ~~1  & 24 &	~1.5773348519927783(~0) &	 2.6(-12) \\
   ~~~1     & ~~~1     &   ~~~1   &   ~~10  &  ~10  & 23 &	~2.4237300030396556(~0) &	 3.1(-12) \\
   ~~~1     & ~~10     &   ~~10   &   ~~10  &  ~10  & 21 &	~3.0275420892666491(~0) &	 7.4(-13) \\
   ~~10     & ~~10     &   ~~10   &   ~~10  &  ~10  & 21 &	~4.0329202866021152(~0) &	 1.6(-12) \\ \hline
   -0.1     & -0.1     &   ~0.1   &   ~0.1  &  0.1  & 29 &	~9.2562395524222385(-1) &	 2.4(-12) \\
   ~0.1     & ~0.1     &   ~~-1   &   ~~-1  &  ~~1  & 33 &	~8.6187455263857027(-1) &	 4.4(-12) \\
   ~~-1     & ~~~1     &   ~~~1   &   ~-10  &  ~10  & 35 &	~1.3353894631528094(~0) &	 4.6(-12) \\
   ~~~1     & ~-10     &   ~-10   &   ~~10  &  ~10  & 28 &	~1.0275704201029547(~0) &	 2.8(-12) \\
   ~-10     & ~-10     &   ~-10   &   ~-10  &  ~10  & 52 &  -2.2446238129792420(~1) &	 2.7(-12) \\ \hline
\end{tabular*}
\label{TABLE: 6}
\end{table}

\begin{figure}[!ht]
\begin{center}
\begin{tabular}{cccc} \includegraphics[width=0.45\textwidth]{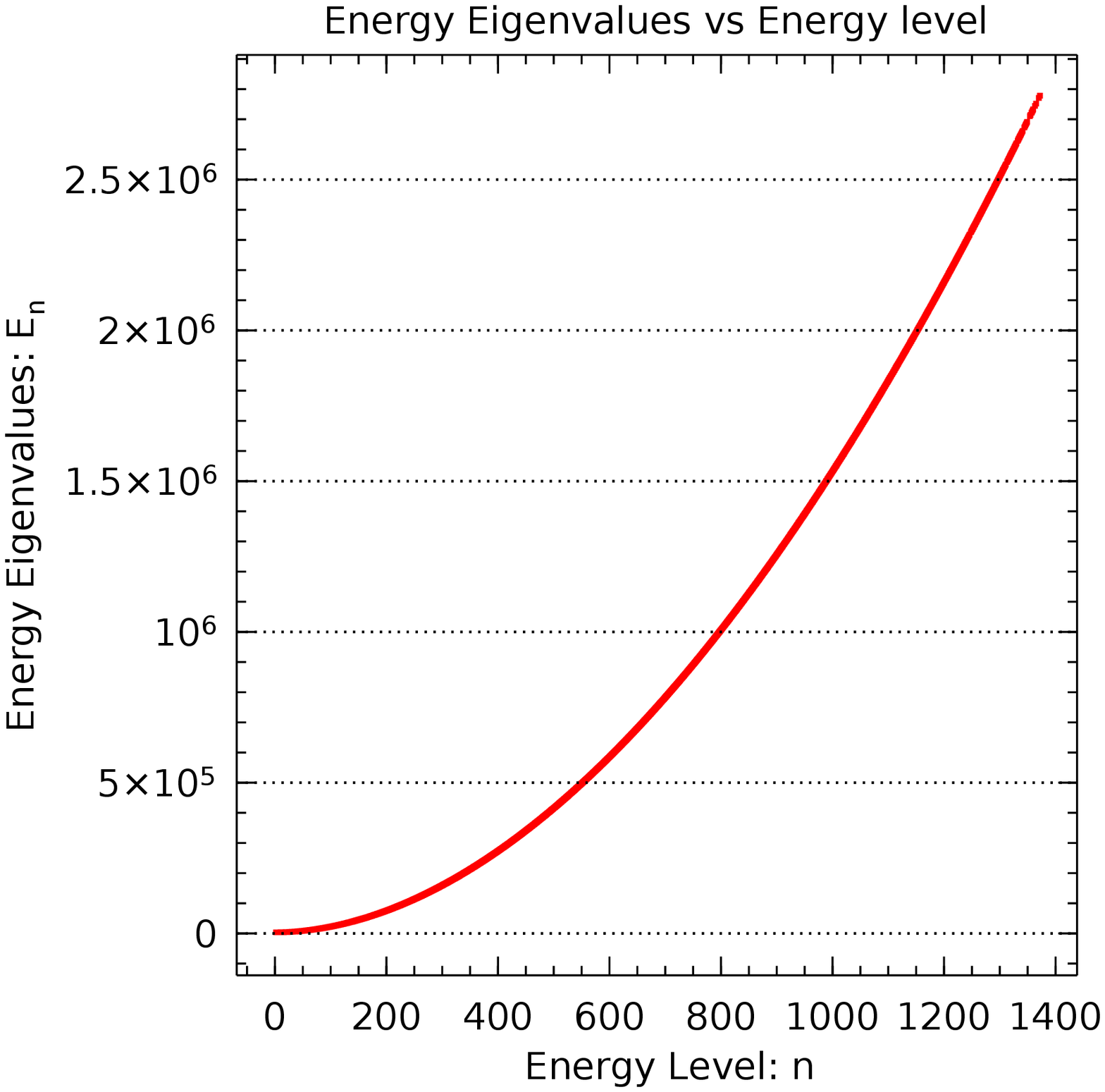} &  \includegraphics[width=0.45\textwidth]{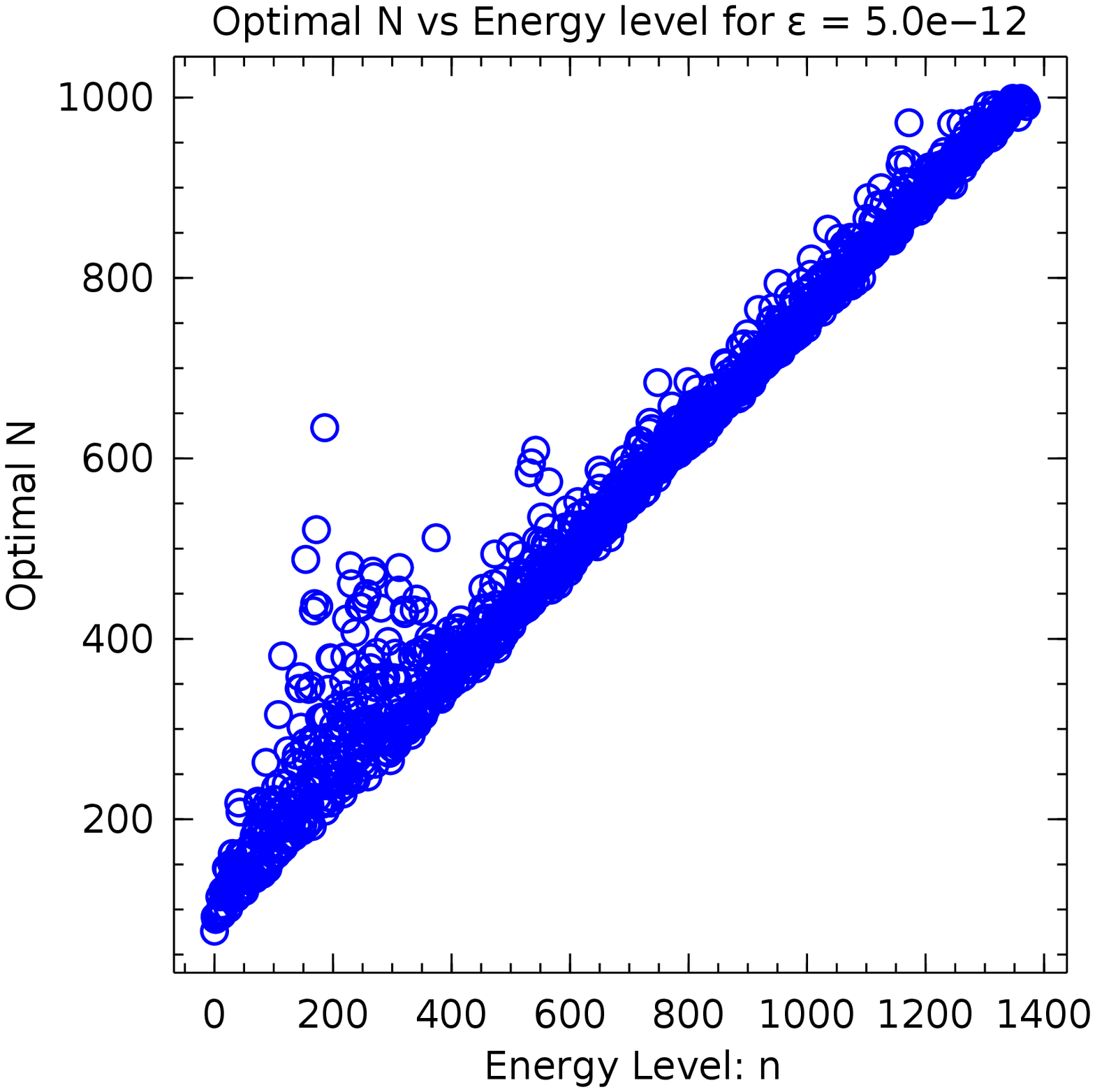} \\
(a) & (b)
\end{tabular}
\caption{Figure (a) displays an approximation for 1353 eigenvalues that achieved a approximate absolute error less than $5\times 10^{-12}$ for the ten well potential $V(x) = T_{20}(x)-1$ as shown in equation \eqref{formula: ten well potential} with $\phi(x) = \sinh(x)$. Figure (b) displays the value of $N$ needed for each eigenvalue in figure (a) to achieve an approximate absolute error less than $5\times 10^{-12}$.}
\label{figure: Multiple wells}
\end{center}
\end{figure}

\begin{figure}[!ht]
\begin{center}
\begin{tabular}{cccc} \includegraphics[width=0.34\textwidth]{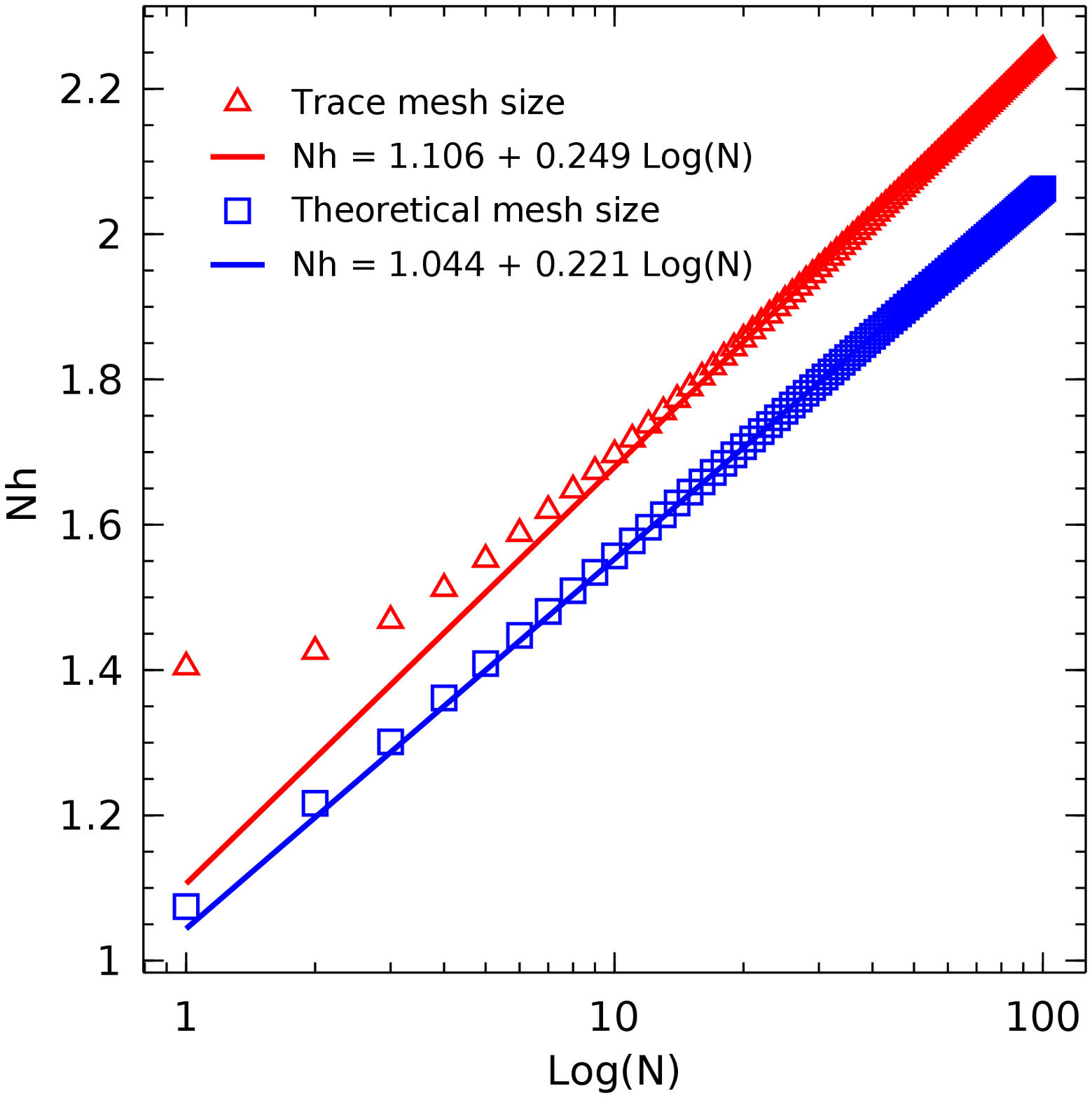} &  \includegraphics[width=0.35\textwidth]{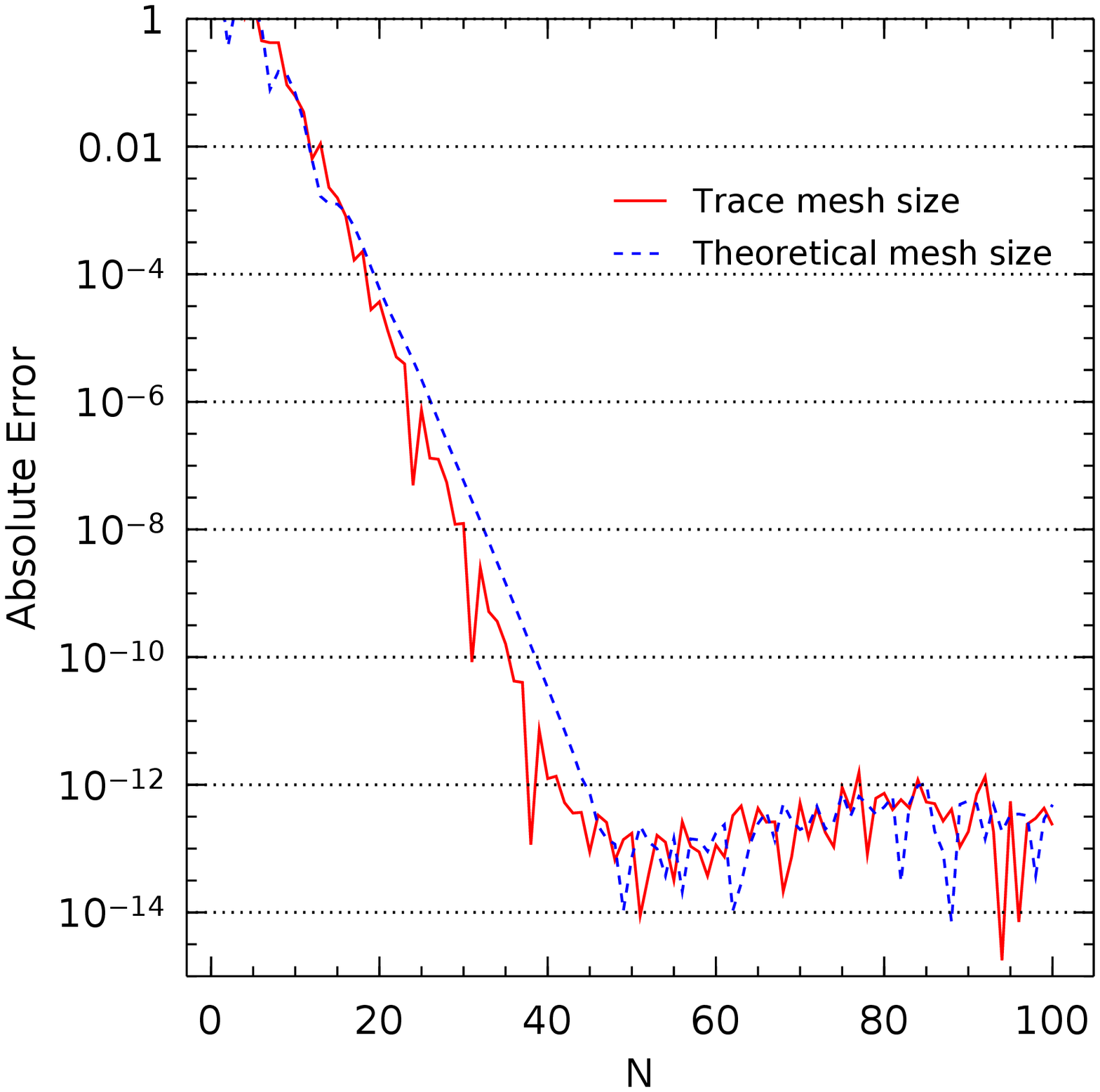} \\
(a) & (b) \\
 \includegraphics[width=0.34\textwidth]{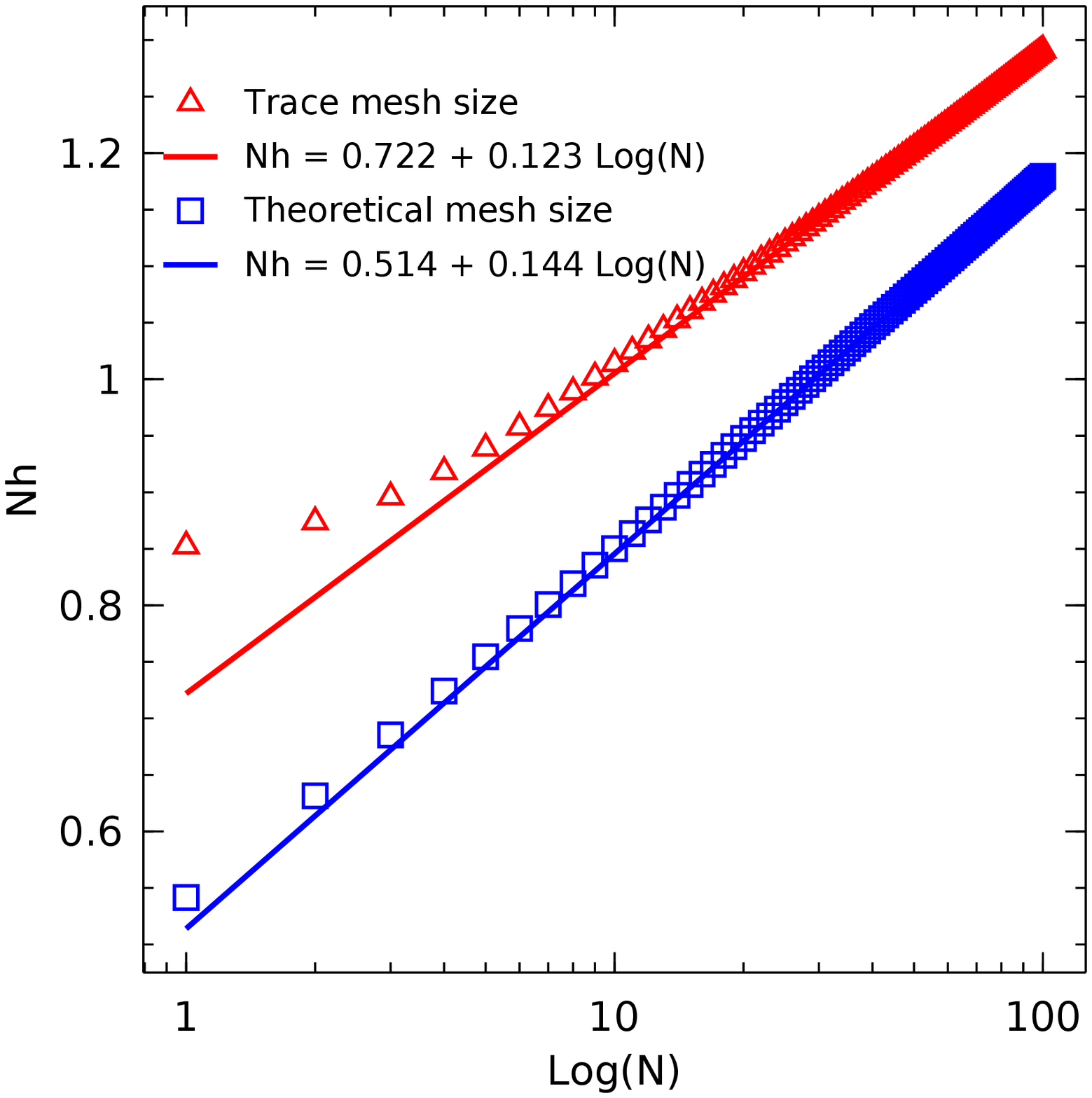} &  \includegraphics[width=0.35\textwidth]{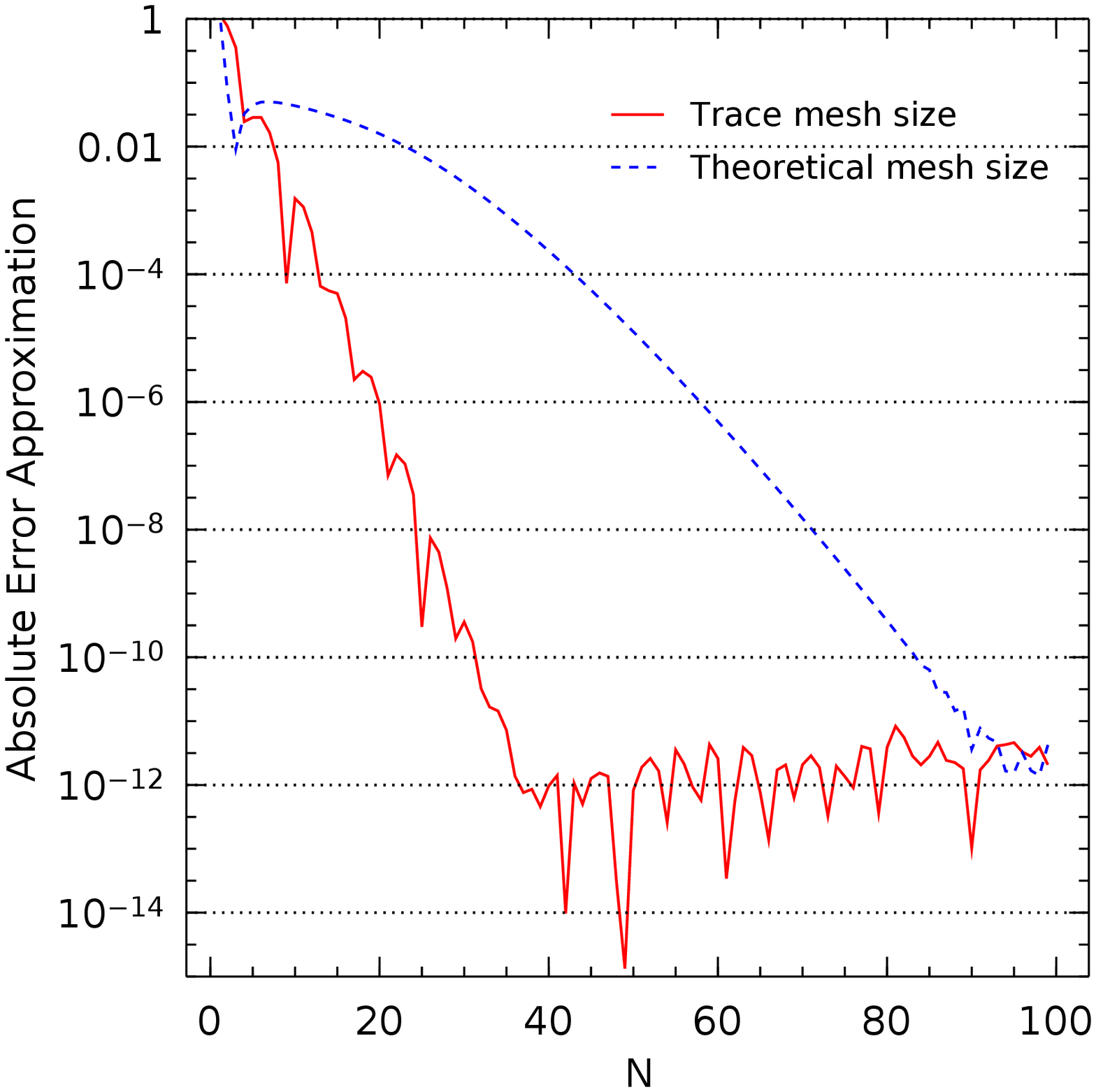} \\
(c) & (d) \\
\includegraphics[width=0.34\textwidth]{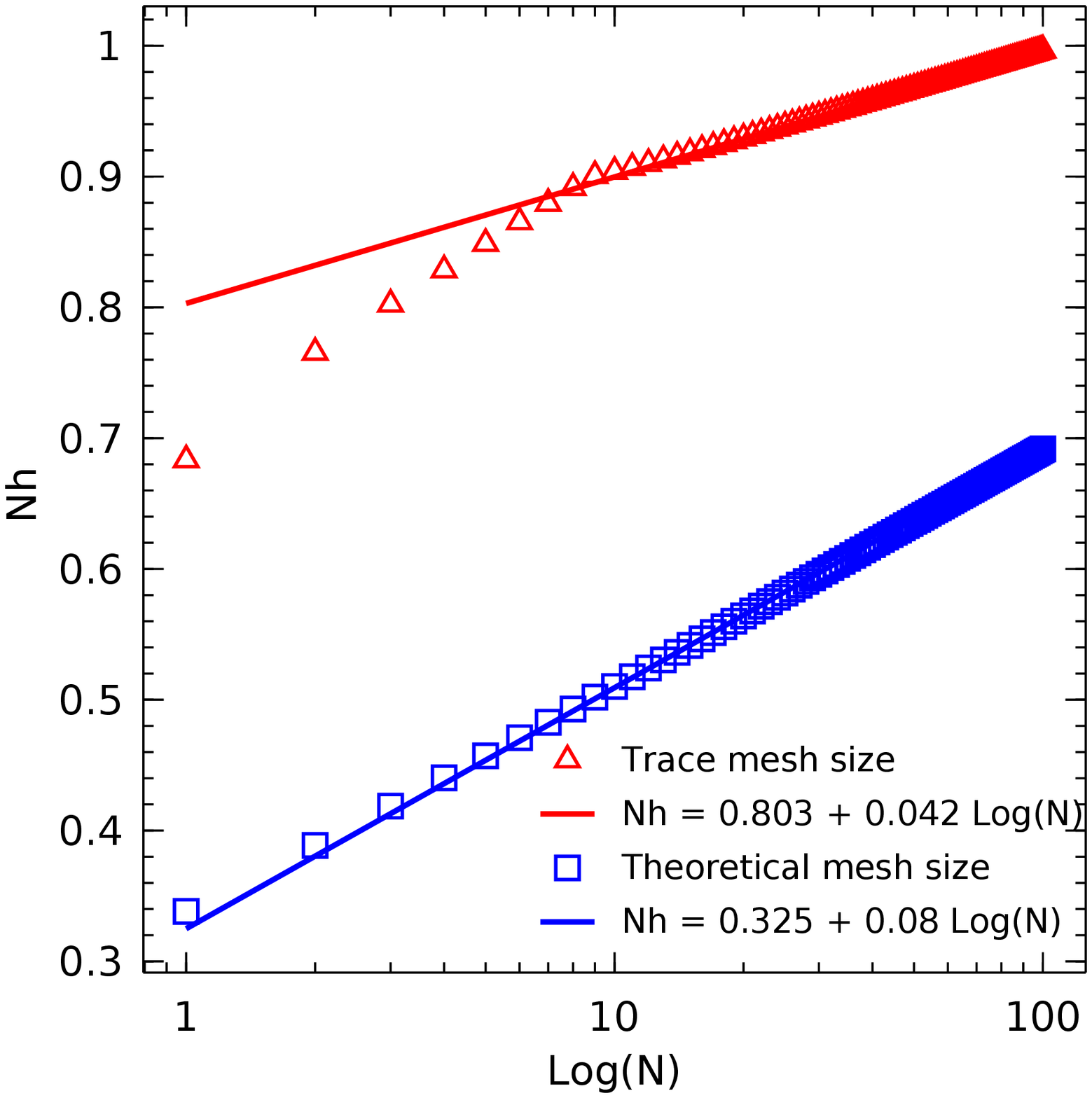} &  \includegraphics[width=0.35\textwidth]{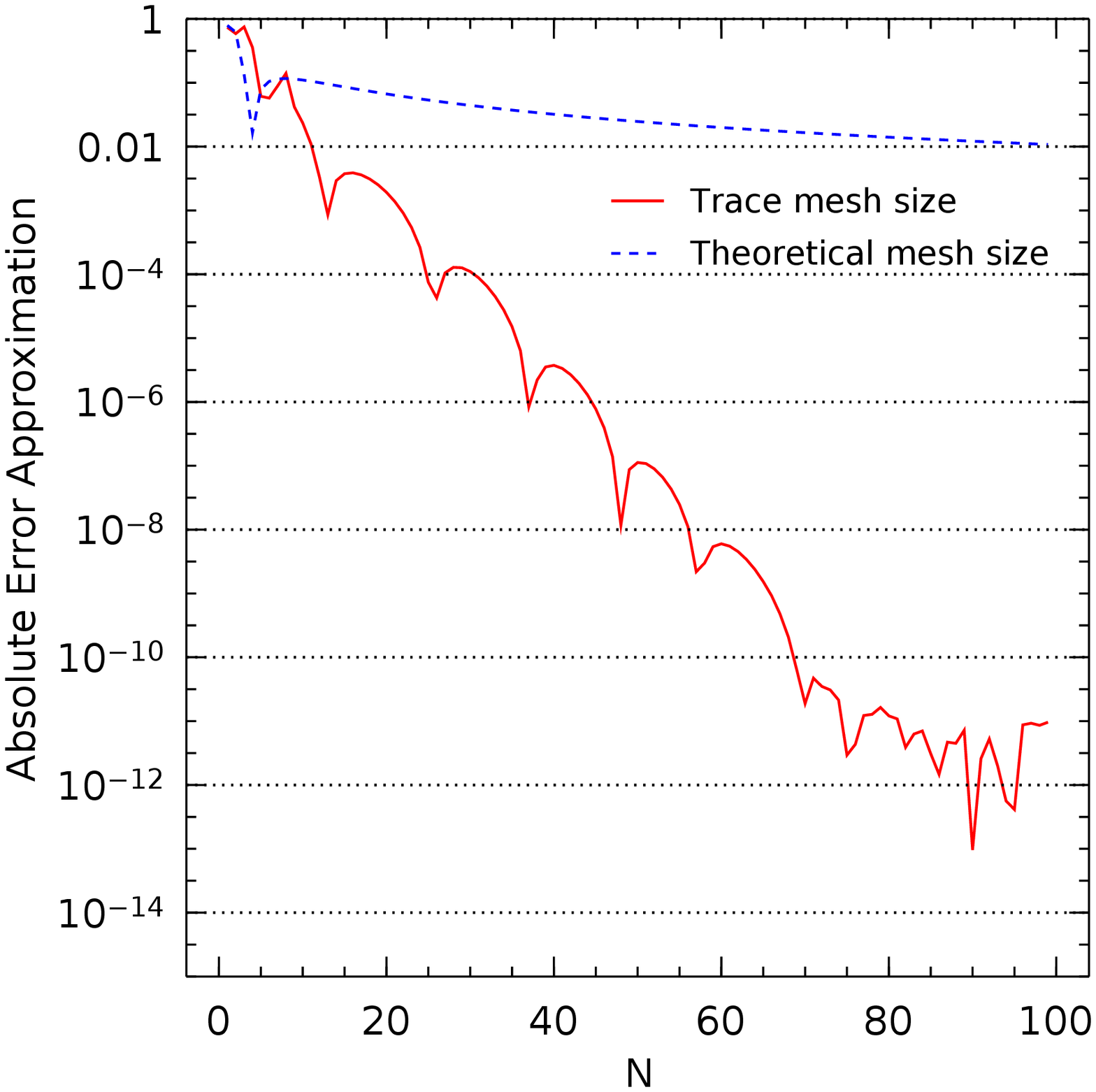} \\
(e) & (f)
\end{tabular}
\caption{Figures (a), (c) and (e) display both the theoretical and optimized trace step sizes used when evaluating the absolute error of the DESINC method in figure (b), (d) and (f) for the potentials $V_{2}(x)  =  4x^2 -6x^4+x^6$ with exact eigenvalue $E_{1} = -9$, $V(x) = T_{10}(x)-1$ and $V(x) = T_{20}(x)-1$ respectively. }
\label{figure: Multiple wells step size}
\end{center}
\end{figure}

\clearpage

\bibliography{My_Bibliography,library}

\end{document}